\documentclass{amsart}
\usepackage{latexsym,amsxtra,amscd,ifthen,amsmath,color, multicol,hyperref}
\usepackage{amsfonts}
\usepackage{verbatim}
\usepackage{amsmath}
\usepackage{amsthm}
\usepackage{amssymb}
\usepackage[notcite,notref,final]{showkeys}
 \usepackage[all,cmtip]{xy}
 \usepackage[normalem]{ulem}

\numberwithin{equation}{section}

\theoremstyle{plain}
\newtheorem{theorem}{Theorem}[section]
\newtheorem{lemma}[theorem]{Lemma}

\newtheorem{proposition}[theorem]{Proposition}

\theoremstyle{definition}
\newtheorem{definition}[theorem]{Definition}
\newtheorem{example}[theorem]{Example}
\newtheorem{notation}[theorem]{Notation}

\newtheorem{construction}[theorem]{Construction}

\newtheorem{remark}[theorem]{Remark}
\newtheorem{question}[theorem]{Question}

\makeatletter              
\let\c@equation\c@theorem  
\makeatother

\DeclareMathOperator{\Ext}{Ext}
\DeclareMathOperator{\gr}{gr} 
\DeclareMathOperator{\Ima}{Im}

\DeclareMathOperator{\Ker}{Ker}

\DeclareMathOperator{\Tor}{Tor}

\newcommand{\DOT}{\setlength{\unitlength}{1pt}\begin{picture}(2.5,2)
               (1,1)\put(2,2.5){\circle*{2}}\end{picture}}

\newcommand{\bu}{\DOT} 

\newcommand{\mc}{\mathcal}
\newcommand{\NN}{\mathbb N}

\newcommand{\coh}{{\rm H}}

\newcommand{\N}{{\mathbb N}}
\newcommand{\id}{\mbox{\rm id}}
\newcommand{\Hom}{\mbox{\rm Hom\,}}

\newcommand{\ot}{\otimes}

\newcommand{\ZZ}{\mathbb{Z}}

\newcommand{\HH}{{\rm HH}}

\newcommand{\ds}{\displaystyle}

\newcommand{\beq}{\begin{equation}}
\newcommand{\eeq}{\end{equation}}

\begin{document}

\title[PBW deformations of smash product algebras from Hopf actions]
{Poincar\'e-Birkhoff-Witt deformations of smash product algebras from Hopf actions on Koszul algebras}

\author{Chelsea Walton and Sarah Witherspoon}

\address{Department of Mathematics, Temple University, Philadelphia, Pennsylvania 19122, USA}
\email{notlaw@temple.edu}

\address{Department of Mathematics, Texas A\&M University, College Station, Texas 77843, USA}

\email{sjw@math.tamu.edu}

\bibliographystyle{abbrv}       

\begin{abstract}
Let $H$ be a Hopf algebra and let $B$ be a Koszul $H$-module algebra. We provide necessary and sufficient conditions for a filtered algebra to be a Poincar\'e-Birkhoff-Witt (PBW) deformation of the smash product algebra $B \#H$. Many examples of these deformations are given.
\end{abstract}



\maketitle


\setcounter{section}{-1}


\section{Introduction}

Given a Hopf algebra ($H$-) action on a  Koszul algebra $B$, the aim of this work is to provide necessary and sufficient conditions for a certain filtered algebra, namely $\mc{D}_{B,\kappa}$ in Notation~\ref{def:D Bkap} below, to be a {\it Poincar\'e-Birkhoff-Witt (PBW) deformation} of  the smash product algebra $B \# H$, i.e, gr$\mathcal{D}_{B,\kappa} \cong B \#H$ [Definition~\ref{def:PBW}]. 
One well-studied case is that of group actions on polynomial rings, where many algebras of interest arise
as such deformations, see 
for example
\cite{CBH}, \cite{Drinfeld:hecke},  \cite{EG:Symp}, \cite{Lusztig:hecke}, \cite{RS}, \cite{SW:DOA}.
For group actions  on other Koszul algebras, see 
\cite{LS}, \cite{NW}, \cite{SW}, \cite{SP}. 
There are some results involving Hopf algebra actions, such as those of Khare  
\cite{Khare} when $H$ is cocommutative and $B$ is a polynomial algebra. More specifically, 
 the case when $H = U({\mathfrak{g}})$, with $\mathfrak{g}$ the Lie algebra of a (not necessarily connected) reductive algebraic group, was studied by Etingof, Gan, and Ginzburg
\cite{EGG} and by Khare and Tikaradze  \cite{TK}
where ${\mathfrak{g}} = {\mathfrak{sl}}_2$.
Results for an  action of the quantized enveloping algebra $H=U_q(\mathfrak{sl}_2)$ on the quantum plane are provided by Gan and Khare \cite{GK}. 

The goal of this paper is to provide a general theorem encompassing all of the above known classes of examples from the literature. Specifically, Theorem~\ref{thm:mainintro} gives PBW deformation
conditions for $B\# H$, and it 
only requires the following of $H$ and $B$: (1) the antipode of the Hopf algebra $H$ is bijective, (2) the Koszul $H$-module algebra $B$ is connected ($B_0 = k$), and (3) the $H$-action on $B$ preserves the grading of $B$.
We then apply our theorem to several different choices of Hopf algebras acting on
Koszul algebras to obtain nontrivial PBW deformations, both known and new.
Our work indicates that such examples abound. 

Many ring theoretic properties are preserved under PBW deformation. To discuss this, let us consider the following definition.

\begin{definition} \label{def:PBW}
Let $D=\bigcup_{i \geq 0} F_i$ be a filtered algebra with $\{0\} \subseteq F_0 \subseteq F_1 \subseteq \cdots \subseteq D$. We say that $D$ is a {\it Poincar\'e-Birkhoff-Witt (PBW) deformation} of an $\N$-graded algebra $A$ if $A$ is isomorphic to the associated graded  algebra $\gr_FD = \bigoplus_{i \geq 0} F_i/F_{i-1}$, as $\N$-graded algebras.
\end{definition}

Now if gr${}_F D$ is an integral domain, prime, or (right) noetherian, then so is $D$. Moreover, if $D$ is affine with the standard filtration $F'$, then the Gelfand-Kirillov (GK) dimensions of $D$ and of gr${}_{F'}D$ are equal; GKdim(gr${}_F D$) $\leq$ GKdim($D$) for a general filtration $F$ of $D$. The Krull dimension and global dimension of gr${}_F D$ serve as upper bounds for the corresponding dimensions of $D$.
These ring theoretic results can all be found in \cite{MR}.
Homological properties preserved under PBW deformation have also been investigated; see \cite{BT} and \cite{WZ} regarding the Calabi-Yau property, for instance.
The representation theory of some classes of PBW deformations of smash product algebras has been thoroughly studied in the literature and still remains an active area of research. Some examples of  PBW deformations whose representation theory is of interest include rational Cherednik algebras, symplectic reflection algebras, and various types of Hecke algebras (see, for example, \cite{Drinfeld:hecke}, \cite{EGG}, \cite{EG:Symp},  \cite{Lusztig:hecke}, \cite{RS}, and for more recent work, see  \cite{DingTsy}, \cite{LosevTsy}, \cite{Tik2010}, \cite{Tsy}) .

In order to state the main result, we need the following notation and terminology. Let $k$ be a field of arbitrary characteristic  and let an unadorned $\otimes$ mean $\otimes_k$. 
Let $\N$ denote the natural numbers, including 0. 
Recall that an $\N$-graded algebra is {\it Koszul} if its trivial module $k$ admits a linear minimal graded free resolution; see \cite[Chapter~2]{PP} for more details.

\begin{notation}
\label{not:B,H,kappa} [$H, B, I, \kappa, \kappa^C, \kappa^L$]  First, let $V$ be a finite dimensional vector space over $k$.
\smallskip

\noindent (i) Let $H$ be a Hopf algebra with the standard structure notation: $(H, m, \Delta, u ,\epsilon, S)$. Here, we assume that the antipode $S$ of $H$ is bijective. 
\smallskip

\noindent (ii) Let $B = T_k(V)/(I)$ be an $\mathbb{N}$-graded, Koszul, left $H$-module algebra $B = \bigoplus_{j \geq 0} B_j$ with $B_0 =k$ and $I\subseteq V\ot V$. We assume that 
the action of $H$ preserves the grading and the subspace $I$ of $V\ot V$. 
So in this case, $V$ is an $H$-module. 
\smallskip

\noindent (iii) Take $\kappa:I \rightarrow H \oplus (V \otimes H)$ to be a $k$-bilinear map,
where $\kappa$ is the sum of its {\it constant} and {\it linear} parts $\kappa^C:I \rightarrow H$ and $\kappa^L:I \rightarrow V \otimes H$, respectively. 
\end{notation}

\begin{notation} \label{def:D Bkap} [$\mc{D}_{B,\kappa}$]  
Let $\mc{D}_{B,\kappa}$ be the filtered $k$-algebra given by
$$\mc{D}_{B,\kappa} = \frac{T_k(V) \# H}{\left(r - \kappa(r)\right)_{r\in I}}.$$
Here, we assign the elements of $H$ degree 0. 
\end{notation}
\smallskip

Our main result is given as follows.

\begin{theorem}[Theorem~\ref{thm:main}] \label{thm:mainintro} 
The algebra $\mc{D}_{B,\kappa}$ is a PBW deformation of $B\# H$ if and only if the following conditions hold:
\begin{enumerate}
\item $\kappa$ is $H$-invariant [Definition~\ref{def:H-inv}]; 
\end{enumerate}
and the following equations hold for the  maps $\kappa^C \otimes \id - \id \otimes \kappa^C$ and $\kappa^L \otimes \id - \id \otimes \kappa^L$, which are defined on the intersection $(I \otimes V) \cap (V \otimes I)$:
\begin{enumerate}
\item[(b)] ${\rm Im}(\kappa^L \otimes \id - \id \otimes \kappa^L) \subseteq I$;

\item[(c)] 
$\kappa^L \circ (\kappa^L \otimes \id - \id \otimes \kappa^L) = -(\kappa^C \otimes \id - \id \otimes \kappa^C)$;

\item[(d)] 
$\kappa^C \circ ( \id \otimes \kappa^L - \kappa^L \otimes \id) \equiv 0.$
\end{enumerate}
\end{theorem}

In the case that $H$ is cocommutative and $B$ is the symmetric algebra $S(V)$,
this result was proven by Khare \cite[Theorem~2.1]{Khare}, via the Diamond Lemma.
Our proof is a generalization of that of Braverman and Gaitsgory \cite[Lemma~0.4, Theorem~0.5]{BG} 
(where $H=k$) and of Shepler and the second author \cite[Theorem~5.4]{SW} (where $H$
is a group algebra). 

Background information on Hopf algebra (co)actions, Hochschild cohomology, and deformations of algebras are provided in Section~\ref{sec:background}. In Section~\ref{sec:resolution}, we produce a free resolution of the smash product algebra $B\#H$; see Construction~\ref{const:X} and Theorem~\ref{thm:resolution}. This resolution is 
adapted from Guccione and Guccione \cite{GG}; Negron independently constructed a similar 
resolution \cite{Negron}. 
Our resolution is used in the proof of Theorem~\ref{thm:mainintro}, which is given in Section~\ref{sec:mainthm}. Many examples of  PBW deformations of $B \#H$ are provided in Section~\ref{sec:examples}, including/involving:
\begin{itemize} 
\item{} [Example~\ref{ex:CBH}] the {\it Crawley-Boevey-Holland algebras}; 
\item {}[Examples~\ref{ex:H8} and \ref{ex:Ha1}] some actions of semisimple, noncommutative, noncocommutative, Hopf algebras on skew polynomial rings;
\item{} [Examples~\ref{ex:T(n)} and~\ref{ex:Hsw}] actions of the Sweedler and the 
  Taft algebras on the polynomial ring $k[u,v]$; 
\item{} [Example~\ref{ex:Uqsl2}] the {\it quantized symplectic oscillator algebras of rank 1}.
\end{itemize}
All of the examples of $B \# H$ above have nontrivial PBW deformations.


\section{Background material} \label{sec:background}

We begin by discussing Hopf (co)actions on algebras and (co)modules and end with a discussion on deformations of algebras. For further background on these topics, we refer the reader to \cite{Montgomery} and \cite{BG,Gerstenhaber},
respectively.

\subsection{Hopf algebra (co) actions}\label{subsec:actions} 

\begin{definition} \label{def:Haction} (i) For a left $H$-module $M$, we denote the $H$-action by $\cdot: H \otimes M \to M$, that is by $h \cdot m \in M$ for all $h \in H$, $m \in M$. Similarly for all $h \in H$ and $m \in M$, we denote the right $h$-action on $m$ by $m \cdot h$.
\smallskip

\noindent (ii) Given a Hopf algebra $H$ and an algebra $A$, we say
that {\it $H$ acts on $A$} (from the left, as a Hopf algebra) if $A$ is a left $H$-module and
$$h \cdot (ab) = \sum (h_1 \cdot a)(h_2 \cdot b)
\text{\hspace{.1in} and \hspace{.1in}} h \cdot 1_A = \epsilon(h)
1_A$$  
for all $h \in H, a,b \in A$, where the comultiplication is given by $\Delta(h) =\sum h_1 \otimes h_2$ (Sweedler's notation). In this case, we say that $A$ is a {\it left $H$-module algebra}.
\smallskip

\noindent (iii) For any left $H$-comodule $M$, we denote the left $H$-coaction by $\rho(M) \subseteq H \otimes M$, where $\rho(m) = \sum m_{-1} \otimes m_0$ for $m_{-1} \in H$ and $m, m_0 \in M$. Likewise, the right $H$-coaction on a right $H$-comodule  $M$ is given by $\rho(m) = \sum m_0 \otimes m_1$ for $m, m_0 \in M$ and $m_1 \in H$.
\end{definition}

Note that $H$ is naturally an $H$-bimodule via left and right multiplication. This yields a {\it left $H$-adjoint action on $H$} given by 
\begin{equation} \label{eq:Hadj0}
h \cdot \ell :=\sum h_1 \ell S(h_2)
\end{equation}
for $h,\ell \in H$. Moreover, if $V$ is a left $H$-module, 
we give $V \otimes H$ an $H$-bimodule structure as follows:
$h  (v \otimes \ell) = \sum (h_1 \cdot  v) \otimes h_2 \ell$ and
$(v \otimes \ell) h = v \otimes \ell h,$
for all $h, \ell \in H$ and $v \in V$.
A {\it left $H$-adjoint action on} $V \otimes H$ arises by combining these:
\begin{equation} \label{eq:Hadj1}
   h \cdot (v\ot \ell) := \sum (h_1 \cdot v) \ot h_2 \ell S(h_3).
\end{equation}
The left $H$-adjoint actions in \eqref{eq:Hadj0} and \eqref{eq:Hadj1} extend to the
standard  left $H$-adjoint action on $A=B \# H$ (where $B=T_k(V)/(I)$ 
as in Notation~\ref{not:B,H,kappa}(ii)), via Definition~\ref{def:Haction},
since the action of $H$ preserves $I$.

\medskip

Now we discuss the $H$-invariance of the map $\kappa$ [Notation~\ref{not:B,H,kappa}(iii)], which is one of the necessary conditions for the filtered algebra $\mc{D}_{B, \kappa}$ [Notation~\ref{def:D Bkap}] to be a PBW deformation of $B \# H$.

\begin{definition} \label{def:H-inv} Recall Notation~\ref{not:B,H,kappa}.
We say that the map $\kappa$ is {\it $H$-invariant} if 
$h \cdot  (\kappa(r)) = \kappa(h \cdot r)$ in $H\oplus (V\ot H)$, for any relation $r \in I$ and $h\in H$. 
\end{definition}

\subsection{Deformations of algebras and Hochschild cohomology}

In this part, we remind the reader of the notion of a deformation of a $k$-algebra $A$ and how Hochschild cohomology plays a role in its construction.
This is  seminal work of Gerstenhaber~\cite{Gerstenhaber},
adapted to our graded setting as in Braverman and Gaitsgory~\cite{BG}. 

\begin{definition} \label{def:deformation} [$A_t$, $A_{(j)}$] 
Let $A$ be an associative algebra and let $t$ be an indeterminate. A {\it deformation}
of $A$ over $k[t]$ is an associative $k[t]$-algebra $A_t$ over $k[t]$, which is isomorphic to $A[t]$ as $k$-vector spaces, with multiplication given by
$$a_1 \ast a_2 = \mu_0(a_1 \ot a_2) + \mu_1(a_1 \ot a_2)t + \mu_2(a_1 \ot a_2)t^2 + \cdots,$$
for all $a_1,a_2 \in A$. Here, $\mu_i: A \ot A \rightarrow A$ is a $k$-linear map, referred to as the {\it $i$-th multiplication map}. Moreover,  $\mu_0(a_1 \ot a_2) =a_1a_2$ is the usual product in $A$.

Now assume that $A$ is graded by $\N$. 
A {\it graded deformation} of $A$ over $k[t]$ is an algebra $A_t$ as above, which is 
itself graded by $\N$, setting deg($t$) =1. The map $\mu_i$ is homogeneous of degree $-i$ in this case.
A {\it $j$-th level graded deformation} of $A$ is a graded associative algebra $A_{(j)}$ over $k[t]/(t^{j+1})$, that is isomorphic to $A[t]/(t^{j+1})$ as $k$-vector spaces, with multiplication given by
$$a_1 * a_2 =\mu_0(a_1 \ot a_2) + \mu_1(a_1 \ot a_2)t + \dots + \mu_j(a_1 \ot a_2)t^j.$$
The maps $\mu_i: A \otimes A \rightarrow A$ are extended to be linear over $k[t]/(t^{j+1})$.
\end{definition}

The associativity of $\ast$ for the deformation $A_t$ imposes conditions on the maps $\mu_i$. Specifically, for each degree $i$, the following equation must hold for all $a_1,a_2,a_3\in A$: 
\begin{equation} \label{eq:mu}
\sum_{j=0}^i \mu_j(\mu_{i-j}(a_1\ot a_2)\ot a_3) = 
\sum_{j=0}^i \mu_j(a_1\ot \mu_{i-j}(a_2\ot a_3)).
\end{equation}
We use Hochschild cohomology to study these equations.
 
\begin{definition} \label{def:Hoch coh} [$\mathbf{B}_{\DOT}(A)$]
Let $A$ be a $k$-algebra and let $M$ be an $A$-bimodule, or equivalently, an $A^e$-module. Here, $A^e:=A \ot A^{op}$. The {\it Hochschild cohomology} of $M$ is 
$\HH^n(A,M) := \Ext_{A^e}^n(A,M).$
Moreover, this cohomology may be derived from the {\it bar resolution} $\mathbf{B}_{\DOT}(A)$ of the $A^e$-module $A$:
$$\hspace{-.7in} \mathbf{B}_{\DOT}(A): \quad \quad\cdots \overset{\delta_3}{\longrightarrow} A^{\ot 4}
\overset{\delta_2}{\longrightarrow} A^{\ot 3}
\overset{\delta_1}{\longrightarrow} A \ot A
\overset{\delta_0}{\longrightarrow} A
\longrightarrow 0$$
where
$$\delta_n(a_0 \ot \cdots \ot a_{n+1}) := \sum_{i=0}^n (-1)^i a_0 \ot \cdots \ot a_i a_{i+1} \ot \cdots \ot a_{n+1}$$
for all $n \geq 0$ and $a_0, \dots, a_{n+1} \in A$. When $M =A$, write $\HH^n(A)$ for $\HH^n(A,A)$. Moreover, if $A$ is ($\N$-)graded, then $\HH^n(A)$ inherits the grading of $A$: If $A= \bigoplus_i A_i$, then  $\HH^n(A) = \bigoplus_i \HH^{n,i}(A)$.
\end{definition}

Note that $\Hom_k(A^{\ot n},A) \cong \Hom_{A^e}(A^{\ot (n+2)}, A)$ since the $A^e$-module $A^{\ot (n+2)}$ is induced from the $k$-module $A^{\ot n}$. 
We will identify these two Hom spaces often without further comment.
Now we make some remarks about the multiplication maps $\mu_i$.

\begin{remark} \label{rem:mu}
Using \eqref{eq:mu} for $i=1$, we see that
\begin{equation} \label{eq:mu1}
\mu_1(a_1 \ot a_2)a_3 + \mu_1(a_1a_2 \ot a_3) = a_1\mu_1(a_2 \ot a_3) + \mu_1(a_1\ot a_2 a_3),
\end{equation}
for all $a_1, a_2, a_3 \in A$.
In other words, $\mu_1$ is a Hochschild 2-cocycle on the bar resolution of $A$, that is, $\delta_3^*(\mu_1) := \mu_1 \circ \delta_3$ vanishes.
(Here we have identified the input $a_1 \ot a_2 \ot a_3$  
with $1 \ot a_1 \ot a_2 \ot a_3  \ot 1$ to apply $\delta_3$, under the identification of
Hom spaces described above.)
 
Next, 
using \eqref{eq:mu} for $i=2$, we see that
\begin{multline*}
\mu_2(a_1 \ot a_2) a_3 + \mu_1(\mu_1(a_1 \ot a_2) \ot a_3) + \mu_2(a_1a_2 \ot a_3)  \\
= a_1 \mu_2(a_2 \ot a_3) + \mu_1(a_1 \ot \mu_1(a_2 \ot a_3)) + \mu_2(a_1 \ot a_2a_3).
\end{multline*}
Therefore 
\begin{equation} \label{eq:mu2}
\delta_3^*(\mu_2)(a_1 \ot a_2 \ot a_3) =  \mu_1(\mu_1(a_1 \ot a_2) \ot a_3) -\mu_1(a_1 \ot \mu_1(a_2 \ot a_3)),
\end{equation}
for all $a_1, a_2, a_3 \in A$.
In other words, $\mu_2$ is a cochain on the bar resolution of $A$ whose coboundary is given by the right-hand side of \eqref{eq:mu2}.

For all $i \geq 1$,  (\ref{eq:mu}) is equivalent to  
\begin{equation} \label{eq:mui}
\delta_3^*(\mu_i)(a_1 \ot a_2 \ot a_3) = \sum_{j=1}^{i-1} \mu_j(\mu_{i-j}(a_1 \ot a_2) \ot a_3)- \mu_j(a_1 \ot \mu_{i-j}(a_2 \ot a_3)).
\end{equation}
That is to say, $\mu_i$ is a cochain on the bar resolution of $A$ whose coboundary is given by the right-hand side of \eqref{eq:mui}.
\end{remark}

\begin{definition}
The right-hand side of \eqref{eq:mui} is the {\it $(i-1)$-th obstruction} of the deformation $A_t$ of $A$ from Definition~\ref{def:deformation}.
An $(i-1)$-th level graded deformation (defined by maps $\mu_1,\ldots,\mu_{i-1}$) 
{\em lifts} to an $i$-th level graded deformation
if there exists a map $\mu_i$ for which $\mu_1,\ldots,\mu_{i-1},\mu_i$ define
an $i$-th level graded deformation. 
\end{definition}

The next proposition makes clear the choice of terminology in the above definition.
Ultimately one is interested in a deformation of $A$ over $k[t]$, and its
specializations at particular values of $t$.
The $i$-th level graded deformations are steps in this direction. 

\begin{proposition} \cite[Proposition~1.5]{BG} \label{prop:lifting}
All obstructions to lifting an $(i-1)$-th level graded deformation to the next level
lie in $\HH^{3,-i}(A)$.
An $(i-1)$-th level deformation lifts to the $i$-th level if and only if its $(i-1)$-th obstruction cocycle is zero in cohomology,
i.e., there is a map $\mu_i$ such that (\ref{eq:mui}) holds for all $a_1,a_2,a_3$ in~$A$. 
\end{proposition}

The connection between graded deformations and PBW deformations is well known; 
the following statement is a consequence of the canonical embedding of $A$
as a $k$-linear direct summand of $A[t]$, with splitting map given by specialization
at $t=0$.

\begin{proposition} \label{prop:PBWA_t}  \cite[Remark~1.4]{BG}
Given a graded algebra $A$ and a graded deformation $A_t$ of $A$, then $A_t$ specialized at $t=1$ is a PBW deformation of $A$. \qed
\end{proposition}

Now we explain that the two notions of deformation of $B\#H$ coincide; recall Notations~\ref{not:B,H,kappa} and~\ref{def:D Bkap}. The following result is well known in
related contexts, but
we include some details for the readers' convenience.

\begin{proposition} \label{prop:doa equiv}
The following statements are equivalent.
\begin{itemize}
\item The algebra $\mc{D}_{B,\kappa}:=  \left(T_k(V) \# H\right)/\left(r - \kappa(r)\right)_{r \in I}$ is a PBW deformation of $B\#H$.
\smallskip
\item The algebra $\mc{D}_{B,\kappa,t} := \left(T_k(V) \# H\right)[t]/\left(r - \kappa^L(r) t - \kappa^C(r) t^2\right)_{r \in I}$ is a graded deformation of $B\# H$ over $k[t]$.
\end{itemize}
\end{proposition}

\begin{proof}
Assume that $\mc{D}_{B,\kappa}$ is a PBW deformation of $B\# H$.
By its definition, $\mc{D}_{B,\kappa,t}$ is an associative algebra, and so we need only see
that it is isomorphic to $B\# H [t]$ as a vector space. 
To  this end, use the PBW property to define a $k$-linear map $\pi: B\# H\rightarrow T_k(V)\#H$
whose composition with the quotient map onto $\mc{D}_{B,\kappa}$ is an isomorphism of
filtered vector spaces. 
Extend $\pi$ to  a $k[t]$-linear map from $B\# H [t]$ to $T_k(V)\# H [t]$. 
Its composition with the quotient map to $\mc{D}_{B,\kappa,t}$ is
an isomorphism of $k$-vector spaces; one sees this by a degree argument. 

Conversely, assume that $\mc{D}_{B,\kappa,t}$ is a graded deformation of $B\# H$ over $k[t]$.
We may specialize to $t=1$ to obtain $\mc{D}_{B,\kappa}$. Now apply Proposition~\ref{prop:PBWA_t} to conclude that $\mc{D}_{B,\kappa}$ is a PBW deformation of $B\#H$.
\end{proof}


\section{Resolutions for smash product algebras} \label{sec:resolution}

In this section, let $A$ denote the smash product $B\# H$, which is an $\mathbb{N}$-graded algebra:  $A = \bigoplus_{j \geq 0} (B_j \otimes H)$. Thus $A_0 \cong H$.
The aim is to construct a free $A^e$-resolution $X_{\bu}$
of the $A^e$-module $A$
from resolutions of $H$ and of $B$ (denoted by $C_{\bu}$ and $D_{\bu}$, respectively). This construction simultaneously generalizes results of Guccione and Guccione \cite{GG}
and of Shepler and the second author \cite[Section~4]{SW}. A similar resolution was constructed independently by Negron \cite{Negron}.

\begin{definition} \label{def:C_.} {}[$C_{\bu}$, $C_i$, $C_i'$]
For $i \geq 0$, let $C_i$ denote the $H^e$-module $H^{\ot (i+2)}$. The left $H$-comodule structure $\rho: C_i\rightarrow H\ot C_i$ is given by
$$
   \rho( h^0\ot h^1\ot\cdots \ot h^{i+1}) := \sum h^0_1\cdots h^{i+1}_1 \ot h^0_2
    \ot\cdots\ot h^{i+1}_2 \in H \otimes C_i
$$
for all $h^0,\ldots, h^{i+1}\in H$.  
For $h \in H$, the left (resp., right) $h$-action on an element $x \in C_i$ is given by left (resp., right) multiplication by $h$ in the leftmost (resp., rightmost) factor of $x$.
Now let
$$
C_{\bu}: \quad  \cdots\rightarrow C_1 \rightarrow C_0\rightarrow H\rightarrow 0
$$
be the bar resolution $\mathbf{B}_{\DOT}(H)$ of $H$ [Definition~\ref{def:Hoch coh}], which is an  $H^e$-free resolution of $H$. 

There is
an isomorphism of free $H^e$-modules $C_i\cong H\ot C_i'\ot H$  where $C_i'=H^{\ot i}$ if $i\geq 1$ and $C_0'=k$. We give each $C_i'$ the $H$-comodule structure inducing that on 
$C_i$ under the usual tensor product of comodules.
\end{definition}

\begin{remark} \label{rem:C_.}
The resolution $C_{\bu}$ satisfies the following conditions.

\noindent (i) The right $H$-action and left $H$-coaction on $C_i$ commute in the sense that
for all $x\in C_i$ and $h\in H$, 
$$
   \sum (x\cdot h)_{-1}\ot (x\cdot h)_0 = \sum x_{-1}h_1\ot (x_0\cdot h_2).
$$
That is, each $C_i$ is a Hopf module (for which the action is a left action
and the coaction is a right coaction). 
\smallskip

\noindent (ii) The differentials are left $H$-comodule homomorphisms. 
\end{remark}

\begin{definition} \label{def:D_.} {}[$D_{\bu}$, $D_i$, $D_i'$]
Recall that $B$ is a Koszul algebra. Let 
$$
   \cdots\rightarrow D_1\rightarrow D_0\rightarrow B\rightarrow 0
$$
be the {\it Koszul resolution} of $B$ as a $B^e$-module: 
$ \ D_0= B\ot B$, $\ D_1=B\ot V\ot B$,
$\ D_2=B\ot I\ot B$, and for each $n\geq 3$, 
$D_i=B\ot D_i'\ot B$ where
$$
  D_i' = \bigcap_{j=0}^{i-2} ( V^{\ot j}\ot I\ot V^{\ot (i-2-j)}).
$$
Each $D_i$ is a subspace of $B^{\ot (i+2)}$, and the differential on the Koszul resolution is the one induced by the canonical embedding of the Koszul resolution into the bar resolution of $B$. 
\end{definition}

\begin{remark} \label{rem:D_.}
The resolution $D_{\bu}$ satisfies the following conditions. 

\noindent (i) Each $B^e$-module $D_i$ is a left $H$-module and 
the differentials are $H$-module homomorphisms. 
\smallskip

\noindent (ii) The left actions of $B$ and of $H$ on $D_i$ are compatible in the sense that 
they induce a left action of $A=B\# H$ on $D_i$. 
\smallskip

\noindent (iii) In addition, the right $B$-action on $D_i$ is
compatible with the left $H$-action on $D_i$ in the sense that for all $h\in H$,
$y =  b^0 \ot y' \ot b^{1} \in D_i$, where $y'\in D_i'$, 
$b^0, b^1$, and $b\in B$,
\[
\begin{array}{rl}
   h\cdot (y\cdot b) ~=~ \sum (h_1\cdot y)\cdot (h_2\cdot b) 
&= \left[\sum (h_1 \cdot b^0) \ot (h_2 \cdot y') \ot (h_{3} \cdot b^{1}) \right] \cdot (h_{4} \cdot b)\\
&= \sum(h_1 \cdot b^0) \ot (h_2 \cdot y') \ot (h_{3} \cdot b^{1})b ~=~ (h \cdot y) \cdot b.
\end{array}
\]

\noindent (iv) Each
$D_i$ is considered to be a left $H$-comodule in a trivial way 
by requiring that it be $H$-coinvariant, that is,
the comodule structure is given by maps $\rho_i: D_i\rightarrow
H\ot D_i$,  where $\rho_i(y)=1\ot y$ for all $y\in D_i$.
The maps $\rho_i$ are maps of left $H$-modules, if we give $H \otimes D_i$
the tensor product $H$-module structure, where the factor $H$ has the 
adjoint
$H$-module structure. See Section~\ref{subsec:actions}. 
\end{remark}

\begin{construction} {}[$X_{\bu}$] \label{const:X}
We wish to combine the two resolutions, $C_{\bu}$ and $D_{\bu}$ from Definitions~\ref{def:C_.} and~\ref{def:D_.}, 
to form a resolution $X_{\bu}$ of $A=B\#H$ by $A$-bimodules, via a tensor product. To that end, we first apply ($A \otimes_H -$) to $C_{\bu}$.
Note that  $A$ is free as a right $H$-module (under multiplication), and that $A\ot _{H}
H\cong A$. The following  sequence of $A\ot H^{op}$-modules
 is therefore exact:  
$$
   \cdots\rightarrow A\ot_{H} C_1\rightarrow A\ot_{H}C_0\rightarrow
      A\rightarrow 0 . 
$$

Similarly, we apply ($- \otimes_B A$) to $D_{\bu}$.
Note that $A$ is free as a left $B$-module, and that $B\ot_B A\cong A$.
The following sequence of $B\ot A^{op}$-modules is therefore exact: 
$$
  \cdots\rightarrow D_1\ot_B A\rightarrow D_0\ot_B A\rightarrow
   A\rightarrow 0.
$$
We will next extend the actions on the modules in each of these two
sequences so that they become $A^e$-modules.
Then, we will  take their tensor product over $A$.

We extend the right $H$-module structure on $A\ot_{H} C_{\DOT}$ 
to a right $A$-module structure by defining a right-action of $B$ on $A \otimes_H C_{\bu}$ as follows. 
For all $a\in A$, $x\in C_i$, $b\in B$, we set
\begin{equation} \label{eq:X.1}
    (a\ot_H x) \cdot b := \sum a (x_{-1}\cdot b)\ot_H x_0.
\end{equation}
This does indeed 
make $A\ot_{H} C_i$ into a right $B$-module, and by combining with the 
right action of $H$, gives 
a right action of $A$ on $A\ot_{H}C_i$. Note that for $x=x^0 \otimes \cdots \otimes x^{i+1} \in C_i$ (with $x^0,\ldots, x^{i+1}\in H$), 
$$\rho(hx) ~=~ \sum (hx)_{-1} \otimes (hx)_0 ~=~
\sum h_1x_1^0  \cdots  x_1^{i+1} \ot h_2  x_2^0 \ot \cdots \ot x_2^{i+1}
~=~ \sum h_1 x_{-1} \ot h_2 x_0.$$
The action is well-defined: 
If $h\in H$, then
$$(ah\ot_H x)\cdot b \overset{\eqref{eq:X.1}}{=} \sum ah (x_{-1}\cdot b) \ot_H x_0 
  = \sum a (h_1 x_{-1}\cdot b)\ot_H h_2x_0 \overset{\eqref{eq:X.1}}{=} (a\ot_H hx)\cdot b.$$
Since the differentials on $C_{\DOT}$ are 
$H$-comodule homomorphisms [Remark~\ref{rem:C_.}(ii)],
this action commutes  with the differentials.

We extend the left $B$-module structure on $D_i\ot_B A$ 
to a left $A$-module structure by defining a left action of
$H$ by 
\begin{equation} \label{eq:X.2}
    h\cdot (y\ot_B a) := \sum (h_1\cdot y)\ot_B  h_2 a
\end{equation}
for all $h\in H$, $y\in D_i$, $a\in A$. 
It is well-defined since for all $h\in H$, $b\in B$, we have by the definitions in
Section~\ref{subsec:actions} that:
\[
\begin{array}{rlll}
h \cdot (yb \otimes_B a) &\overset{\eqref{eq:X.2}}{=} \sum (h_1 \cdot (yb)) \otimes_B (h_2 a S(h_3)) 
                                   &= \sum (h_1 \cdot y)(h_2 \cdot b) \otimes_B h_3 a S(h_4)&\\
                                   &= \sum (h_1 \cdot y) \otimes_B (h_2 \cdot b) h_3 a S(h_4)
                                   &= \sum (h_1 \cdot y) \otimes_B (h_2 \cdot (ba))
                                  &\overset{\eqref{eq:X.2}}{=} h \cdot (y \otimes_B ba).
\end{array}
\]

The left $H$-action on
$D_i$ is compatible with the right $B$-action on $D_i$ by Remark~\ref{rem:D_.}(iii).
Again this action commutes with the differentials
since the differentials on $D_{\DOT}$ are
$H$-module homomorphisms [Remark~\ref{rem:D_.}(i)]. 

We may now consider 
$A\ot_{H}C_{\DOT}$ and $D_{\DOT}\ot_B A$ to be 
complexes of $A^e$-modules via the $A$-bimodule structure defined above. 
We  take their tensor product over $A$, letting 
$
   X_{\DOT,\DOT} := (A\ot_{H}C_{\DOT})\ot_A (D_{\DOT}\ot_B A),
$
that is, for all $i,j\geq 0$, 
\begin{equation}\label{xij}
 X_{i,j}:= (A\ot_{H}C_i)\ot_A (D_j\ot _{B} A),
\end{equation}
with horizontal and vertical differentials
$$
   d^h_{i,j}: X_{i,j}\rightarrow X_{i-1,j} \ \ \ \mbox{ and }
   \ \ \ d^v_{i,j}: X_{i,j}\rightarrow X_{i,j-1}
$$
given by 
$  d^h_{i,j}:= d_i^{C_{\DOT}} \ot \id$ and 
$d^v_{i,j}:= (-1)^i\, \id\ot d_j ^{D_{\DOT}}$. 

Finally, let $X_{\DOT}$ be the total complex of $X_{\DOT, \DOT}$: 
\begin{equation}\label{resolution-X}
  \cdots\rightarrow X_2\rightarrow X_1\rightarrow X_0\rightarrow A\rightarrow 0,
\end{equation}
 with $X_n = \oplus_{i+j=n} X_{i,j}$, where $X_{i,j}$ is defined in (\ref{xij}). 
\end{construction}

\begin{theorem}\label{thm:resolution}
We have the following statements.
\begin{enumerate}
\item For each $i,j$, the $A^e$-module $X_{i,j}$ is isomorphic to $A\ot C_i'\ot D_j'\ot A$.
\item The complex $X_{\DOT}$ given in (\ref{resolution-X}) is a free
resolution of the $A^e$-module $A$.
\end{enumerate}
\end{theorem}

\begin{proof} 
(a) Write $C_i \cong H\ot C_i'\ot H$ and $D_j\cong B\ot D_j'\ot B$
for vector spaces $C_i'$ and $D_j'$,
as in Definitions~\ref{def:C_.} and~\ref{def:D_.}. Then 
\begin{eqnarray*}
   X_{i,j} & \cong &  (A\ot_{H}H\ot C_i' \ot H) \ot_A 
   (B\ot D_j'\ot B\ot_B A) \\
   & \cong & (A\ot C_i'\ot H)\ot_A (B\ot D_j'\ot A).
\end{eqnarray*}
We will show that this is isomorphic to $A\ot C_i'\ot D_j'\ot A$ as an
$A^e$-module. First define a map as follows:
\begin{eqnarray} \label{eq:resn}
\begin{array}{rll}
  (A\ot C_i'\ot H)\times (B\ot D_j'\ot A) &\rightarrow  
       & A\ot C'_i\ot D'_j\ot A\\
  (a\ot x\ot h , \ b\ot y \ot a') &\mapsto 
    & \sum a ( x_{-1} h_1\cdot b) \ot x_0 \ot (h_2\cdot y) \ot  h_3 a' 
\end{array}
\end{eqnarray}
for all $a,a'\in A$, $x\in C_i'$, $y\in D_j'$, $h\in H$, $b\in B$.
This map is $k$-bilinear by its definition, and we will check
that it is $A$-balanced. First let $b'\in B$. We rewrite $(a \ot x \ot h) \cdot b'$ as follows. First, using $A \ot C_i' \ot H \cong A \ot_H C_i$, identify this element with $a \ot_H  (1 \ot x \ot h) \in A \ot_H C_i$. By \eqref{eq:X.1}, 
$$(a \ot_H  (1 \ot x \ot h)) \cdot b' = \sum a((1 \ot x \ot h)_{-1} \cdot b') \ot_H (1 \ot x \ot h)_0.$$
By Definition~\ref{def:C_.}, and by identifying $x \in C_i'$ with $x^1 \ot x^2 \ot \cdots \ot x^i$, we have that
\[
\begin{array}{rl}
\rho(1 \ot x \ot h) &= \displaystyle \sum (1 \ot x \ot h)_{-1} \ot (1 \ot x \ot h)_0\\
&=\displaystyle \sum (x_1^1 x_1^2 \cdots x_1^i h_1) \ot (1 \ot x_2^1 \ot x_2^2 \ot \cdots \ot x_2^i \ot h_2).
\end{array}
\]
So, $(1 \ot x \ot h)_{-1} = x_{-1}h_1$ and $(1 \ot x \ot h)_0 = x_0 \ot h_2$.
Now $C_i\cong H\ot C_i'\ot H$ as an $H$-comodule, so
\begin{eqnarray*}
  ((a\ot x\ot h)\cdot b', \ b\ot y\ot a') & = & 
   \sum (a ( x_{-1}h_1 \cdot  b')\ot x_0\ot h_2 , \ b\ot y\ot a')\\
     &\mapsto & \sum a ( x_{-2}h_1 \cdot b') 
     (x_{-1}h_2  \cdot  b) \ot x_0\ot  (h_3  \cdot y) \ot h_4 a' .
\end{eqnarray*}
On the other hand,
\begin{eqnarray*}
  ((a\ot x\ot h, \ b'\cdot (b\ot y\ot a')) &=& 
   (a\ot x\ot h, \ b'b\ot y\ot  a')\\
   &\mapsto & \sum a ( x_{-1}h_1\cdot (b'b))  \ot x_0
             \ot (h_2\cdot y) \ot h_3  a'),
\end{eqnarray*}
which is the same as the previous image since $B$ is an $H$-module algebra. 
Now let $\ell\in H$. Then
\begin{eqnarray*}
    ((a\ot x\ot h)\cdot \ell , ~ b\ot y\ot a') &=&
     (a\ot x\ot h\ell ,~ b\ot y \ot a')\\
   & \mapsto & 
     \sum a (x_{-1} h_1\ell_1 \cdot b) \ot x_0 \ot 
    (h_2\ell_2 \cdot y) \ot h_3\ell_3 a' .
\end{eqnarray*}
On the other hand,
\begin{eqnarray*}
    (a\ot x\ot h, ~\ell\cdot (b\ot y\ot a')) &=&
      \sum (a\ot x\ot h, ~(\ell_1 \cdot b) \ot (\ell_2\cdot y) \ot \ell_3 a')\\
     &\mapsto & \sum a(x_{-1} h_1\ell_1 \cdot b) \ot x_0 \ot (h_2\ell_2\cdot y)
     \ot h_3\ell_3 a' ,
\end{eqnarray*}
which is the same as the previous image. 
Therefore, there is an induced map 
$$(A\ot C_i'\ot H)\ot_A (B\ot D_j'\ot A)
\rightarrow A\ot C_i'\ot D_j'\ot A.$$ 
Now, we verify that the map below is an inverse map of \eqref{eq:resn}:
\begin{equation} \label{eq:inverse}
   a\ot x\ot y\ot a' \mapsto (a\ot x\ot 1)\ot_A (1\ot y\ot a').
\end{equation}
It is clear that first applying \eqref{eq:inverse} then \eqref{eq:resn} yields the identity map on ${A \ot C_i' \ot D_j' \ot A}$. On the other hand, the image of first applying \eqref{eq:resn}, then \eqref{eq:inverse} to $(a \ot x \ot h,~b \ot y \ot a')$ is:\\

$
\hspace{.2in} \ds \sum (a(x_{-1}h_1 \cdot b) \ot x_0 \ot 1) \ot_A (1 \ot (h_2 \cdot y) \ot h_3 a')
$
\[
\begin{array}{rl}
{\hspace{1in}} &~= \ds \sum (a(x_{-1}h_1 \cdot b) \ot x_0 \ot 1) \ot_A (\epsilon(h_2) \ot (h_3 \cdot y) \ot h_4 a')\\
&~= \ds \sum (a(x_{-1}h_1 \cdot b) \ot x_0 \ot 1) \ot_A (h_2 \cdot ( 1 \ot  y \ot a'))\\
&~= \ds \sum (a(x_{-1}h_1 \cdot b) \ot x_0 \ot h_2) \ot_A ( 1 \ot  y \ot a')\\
&\overset{\eqref{eq:X.1}}{=}  ((a \ot x \ot h) \cdot b) \ot_A ( 1 \ot  y \ot a')\\
&~=  (a \ot x \ot h) \ot_A ( b \ot  y \ot a').\\
\end{array}
\]
Therefore the two $A^e$-modules, $X_{ij}$ and $A \otimes C_i' \otimes D_i' \otimes A$, are isomorphic as claimed.

(b) We wish to apply the K\"unneth Theorem to show that the complex $X_{\DOT}$ is a free resolution of the $A^e$-module $A$. To that end, 
we check that each term in the complex $D_{\DOT}\ot_{B}A$ is a 
free left $A$-module, and that the image of each differential
in the complex is also projective as a left $A$-module.
First write each $D_i\ot_B A\cong (B\ot D_i'\ot B)\ot_B A\cong
B\ot D_i'\ot A$. 
Define a $k$-linear map $f: A\ot D_i'\ot B\rightarrow B\ot D_i'\ot A$ by
$$
    f(rh\ot y\ot b) = \sum r \ot (h_1\cdot y) \ot h_2 b,
$$
for $h \in H$, $y \in D_i'$, and $r,b \in B$.
Give $A\ot D_i'\ot B$ the structure of a left $A$-module by requiring
$A$ to act by left multiplication on the leftmost factor. 
Clearly this is a free left $A$-module.
The map $f$ is an $A$-module homomorphism by the definition of the
left $A$-action on $B\ot D_i'\ot A$; see \eqref{eq:X.2}. 
We claim that the following map is an inverse map, so that $f$ is an
isomorphism of $A$-modules: Let $S^{-1}$ denote the (composition) inverse
of the antipode $S$ of $H$. 
Let $g: B\ot D_i'\ot A\rightarrow A \ot D_i'\ot B$ be the $k$-linear
map defined by 
$$
    g(r\ot y\ot hb) = \sum rh_2\ot (S^{-1}(h_1)\cdot y) \ot b.
$$
Since for each $h\in H$, we have $\sum h_2 S^{-1}(h_1)=
\varepsilon(h) = \sum S^{-1}(h_2)h_1$
(see e.g.\ \cite[Proposition 7.1.10]{Radford}),
the function $g$ is indeed the inverse of $f$. 
Thus, each term in the complex $D_{\DOT} \ot_B A$ is a free left $A$-module.

That the image of each differential is projective as a left $A$-module
may be proved inductively, starting on one end of the complex
\[
\xymatrix{
   \cdots \ar[r] & D_1\ot_B A \ar[r]^{d_1 \ot id} & D_0\ot _B A \ar[r]^{\hspace{.2in}d_0 \ot id} & A \ar[r] & 0,
}
\]
as follows. 
Since $A$ is a projective left $A$-module and $d_0\ot\id$ is surjective,
the map splits, implying that $\Ker (d_0\ot\id) = \Ima (d_1\ot\id)$ is a direct
summand of the free left $A$-module $D_0\ot _B A$.
Therefore it is projective.
Again, since $\Ima(d_1\ot \id)$
 is projective, the map $d_1\ot \id$ from $D_1\ot_B A$ to
its image splits so that $\Ker (d_1\ot\id) = \Ima (d_2\ot\id)$ is a direct
summand of the free left $A$-module $D_1\ot _B A$. Continuing in this way,
we see that $\Ima (d_i \ot\id)$ is a free left $A$-module for each $i$. 

The K\"unneth Theorem \cite[Theorem 3.6.3]{W}
then gives for each $n$ an exact sequence:
$$
 0\longrightarrow \bigoplus_{i+j=n} \coh_i(A\ot_{H}C_{\DOT})\ot_A 
  \coh_j(D_{\DOT}\ot_B A)
\longrightarrow \coh_n((A\ot_{H}C_{\DOT})\ot_A (D_{\DOT}\ot_B A))$$

\vspace{-.4cm}

$$ \hspace{4.45cm}
   \longrightarrow \bigoplus_{i+j=n-1} \Tor^A_1(\coh_i(A\ot_{H} C_{\DOT}),
  \coh_j (D_{\DOT}\ot_B A))\rightarrow 0.
$$
Now $A\ot_{H}C_{\DOT}$ and $D_{\DOT}\ot _B A$ are exact other than in
degree 0, where their homologies are each $A$: That is, $H_0(A \ot_H C_{\DOT}) = A$ and 
 $H_0(D_{\DOT} \ot_B A)  = A$.
Therefore the only potentially nonzero Tor term is when $i=0=j$,  or Tor$_1^A(A,A)$, yet this equals 0
 since $A$ is flat over $A$.
So for each $n$ we have
$$
  \coh_n((A\ot_{H}C_{\DOT})\ot_A (D_{\DOT}\ot_B A))\cong \bigoplus_{i+j=n}
   \coh_j(A \ot_{H} C_{\DOT})\ot_A \coh_i (D_{\DOT}\ot_B A).
$$
Again the right side is only nonzero when $i=0=j$, so that we have
$$
  \coh_0((A\ot_{H} C_{\DOT})\ot _A (D_{\DOT}\ot_B A)) ~\cong~  
  \coh_0(A\ot_{H}C_{\DOT})\ot_A
   \coh_0 (D_{\DOT}\ot_B A)~\cong~ A\ot_A A ~\cong~ A,
$$
and $\coh_n ((A\ot_{H}C_{\DOT})\ot_A (D_{\DOT}\ot_B A))=0$ for all $n>0$.
Thus we have proven that $X_{\DOT}$ is an $A^e$-free resolution of $A$.
\end{proof}

We next relate the resolution $X_{\bu}$ of $A$ (from Construction~\ref{const:X}) to the bar resolution $\mathbf{B}_{\DOT}(A)$ of $A$. 

\begin{lemma}\label{lem:maps}
There exist degree-preserving 
chain maps between $X_{\DOT}$ and the bar resolution $\mathbf{B}_{\DOT}(A)$ of $A$:
\[
\phi_{\DOT}: X_{\DOT}\longrightarrow \mathbf{B}_{\DOT}(A)
\text{ \hspace{.3in}  and  \hspace{.3in} } \psi_{\DOT}: \mathbf{B}_{\DOT}(A) \longrightarrow X_{\DOT},
\]  
such that $\psi_n \phi_n$ is the identity map on the
$A^e$-submodule $X_{0,n}$ of $X_n$, for each $n\geq 0$. 
\end{lemma}

\begin{proof}
Recall by Notation~\ref{not:B,H,kappa} that $B$ is generated by the vector space $V$,
with quadratic relations $I\subseteq V\ot V$. 
First we prove by induction on $n$  that 
there are degree-preserving maps 
$\phi_n:X_n\rightarrow A^{\ot(n+2)}$
and $\psi_n:A^{\ot(n+2)} \rightarrow X_n$
commuting with the differentials.
For clarity, we denote the 
differential on the bar resolution of $A$ by $\delta$. We have the following diagram:

\[\hspace{-.7in}
\xymatrix{
X_{\DOT}:& \cdots \ar[r] &X_3 \ar[r]^{d_3} \ar@<1.ex>[d]|{\phi_3} 
&X_2 \ar[r]^{d_2} \ar@<1.ex>[d]|{\phi_2} 
&X_1 \ar[r]^{d_1} \ar@<1.ex>[d]|{\phi_1}  
&X_0 \ar[r]^{d_0} \ar@<1.ex>[d]|{\phi_0} &A \ar[r]  &0\\
\mathbf{B}_{\DOT}(A):& \cdots \ar[r] &A^{\ot 5} \ar[r]^{\delta_3} \ar@<1.ex>[u]|{\psi_3} 
&A^{\ot 4} \ar[r]^{\delta_2} \ar@<1.ex>[u]|{\psi_2} 
&A^{\ot 3} \ar[r]^{\delta_1} \ar@<1.ex>[u]|{\psi_1}  
&A\ot A \ar[r]^{\delta_0} \ar@<1.ex>[u]|{\psi_0} &A \ar[r] &0,
}
\]
\medskip

\noindent where $\mathbf{B}_n(A) = A^{\otimes (n+2)}$ and $X_n = \bigoplus_{i+j =n} X_{i,j}$, with $X_{i,j}$ defined in \eqref{xij}; see also Theorem~\ref{thm:resolution}(a).

Define $\phi_0 = \id\ot\id = \psi_0$, the identity map
from $A\ot A$ to itself.
We wish to  define $\phi_{\DOT}$
so that when restricted to $X_{0,\DOT}$, it corresponds to the standard
embedding of the Koszul complex into the bar complex:
For $n=1$, this is the embedding of $A\ot V\ot A$ into $A\ot A\ot A$ via
the containment of $V$ in $A$. 
We may define $\phi_1$ on 
$X_1 = X_{0,1}\oplus X_{1,0} \cong
(A\ot V\ot A) \oplus (A\ot H\ot A)$  by
$\phi_1(1\ot v\ot 1) = 1\ot v\ot 1$ and $\phi_1(1\ot h\ot 1)=1\ot h\ot 1$
for all $v\in V$, $h\in H$.
Note that for $n\geq 2$,
\begin{equation}\label{X0n}
   X_{0,n}\cong A\ot \left( \bigcap _{i=0}^{n-2} V^{\ot i}
   \ot I\ot V^{\ot (n-i-2)}\right) \ot A\, ,
\end{equation}
which is a free $A^e$-submodule of  $A^{\ot (n+2)}$.
For each $i,j$ with $i+j=n$, 
choose a basis of the vector space $C_i'\ot D_j'$
(whose elements are homogeneous of degree $j$,  as $H$ is declared to have degree 0).
By hypothesis, $\phi_{n-1}$ is degree-preserving, and $d_n$ is degree-preserving by construction.
So, applying 
$\phi_{n-1}d_n$ to these basis elements of $C_i' \otimes D_j'$ produces
elements of degree $j$
 in the kernel of $\delta_{n-1}$,
that is, the image of $\delta_n$.
We define $\phi_n$ by choosing (arbitrary) 
corresponding elements in the inverse image of $\Ima(\delta_n)$. 
If we start with an element in $X_{0,n}$, we may choose its image in
$A^{\ot (n+2)}$ under the canonical embedding of $X_{0,n}$
into $A^{\ot (n+2)}$ (see (\ref{X0n})). 
Given $X_{i,j}$ and $X_{i',j'}$ with $i+j=i'+j'=n$ and $i\neq 0$, $i'\neq 0$,  
elements of $X_{i,j}$ have degree $j$ and elements of $X_{i',j'}$ have degree $j'$. So their images under $\phi_n$ may be chosen
independently, and in particular, independently of those of $X_{0,n}$. 
Thus, we have the maps $\phi_n$ as desired.

Now we show that $\psi_n$ may be chosen so that $\psi_n\phi_n$ is the
identity map on $X_{0,n}$. In degree 1, we have summands 
$X_{0,1}\cong A\ot V\ot A$ and
$X_{1,0}\cong A\ot H\ot A$. 
Note that $V\oplus H$
is a direct summand of $A$ as a vector space. We may therefore define
$\psi_1(1\ot v\ot 1)= 1\ot v\ot 1$ in $X_{0,1}$ for all $v\in V$, 
and $\psi_1(1\ot h\ot 1) = 1\ot h\ot 1$ in $X_{1,0}$ for all $h\in H$. We also have that $\psi_1$ is the identity map on elements of the form $1\ot z\ot 1$, for $z$ ranging
over a basis of a 
chosen complement of $V\oplus H$ as a vector subspace of $A$. This complement
may be chosen 
arbitrarily subject to the condition that $d_1\psi_1(1\ot z\ot 1)
=\psi_0\delta_1(1\ot z\ot 1)$. 
Since $\psi_0, d_1, \delta_1$ all have degree 0 as maps, one may
also choose $\psi_1$ to have degree 0. In particular, note that 
$\psi_1\phi_1$ is the identity map on $X_{0,1}$.
Now let $n\geq 2$ and assume that $\psi_{n-2}$ and $\psi_{n-1}$ have
been defined to be degree 0 maps for which $d_{n-1}\psi_{n-1}=
\psi_{n-2}\delta_{n-1}$ and $\psi_{n-1}\phi_{n-1}$ is the identity map
on $X_{0,n-1}$. To define $\psi_{n}$,
first note that $A^{\ot (n+2)}$ contains
the space $X_{0,n}$ as an $A^e$-submodule 
(see (\ref{X0n})) and the image of each $X_{i,j}$
under $\phi_n$ ($n=i+j$, $i\geq 1$).
By construction, their images intersect in 0, the image of $X_{0,n}$
under $\phi_n$ is free, and moreover $\phi_n$ is injective on
restriction to $X_{0,n}$.  
Choose a set of free generators of $\phi_n(X_{0,n})$, 
and choose a set of free generators
of its complement in $A^{\ot (n+2)}$. 
For each chosen generator $x$ of $X_{0,n}$, we 
define $\psi_n (\phi_n (x))$ to be $x$. 
On the complement of $\phi_n(X_{0,n})$, define $\psi_n$ arbitrarily
subject to being a chain map of degree 0.  Thus, $\psi_n \phi_n$ is the identity map on $X_{0,n}$.
Now for all $x \in X_{0,n}$, since 
$d_{n}(x)\in X_{0,n-1}$, we have that
 $\psi_{n-1}\phi_{n-1}d_{n}(x)= d_{n}(x)$, by induction.  
As $\delta_n\phi_n(x)=\phi_{n-1} d_{n}(x)$,
it follows that 
$d_n\psi_n\phi_n(x) = \psi_{n-1}\delta_n \phi_n(x)$. 
So $\psi_n$ also extends the chain map from degree $n-1$ to degree $n$,
as desired.
\end{proof}


\section{Poincar\'e-Birkhoff-Witt Theorem for Hopf algebra actions} \label{sec:mainthm}

Consider the algebra $\mc{D}_{B, \kappa}$ from Notation~\ref{def:D Bkap}. The goal of this section is to prove our main result, Theorem~\ref{thm:mainintro}. We provide necessary and sufficient conditions for $\mc{D}_{B,\kappa}$ to be a PBW deformation of $B \# H$ [Definition~\ref{def:PBW}] as follows.

\begin{theorem}  \label{thm:main}
The algebra $\mc{D}_{B,\kappa}$ is a PBW deformation of $B \# H$ if and only if the following conditions hold:
\begin{enumerate}
\item $\kappa$ is $H$-invariant [Definition~\ref{def:H-inv}];
\end{enumerate}
and the following equations hold for the  maps $\kappa^C \otimes \id - \id \otimes \kappa^C$ and $\kappa^L \otimes \id - \id \otimes \kappa^L$, which are defined on the intersection $(I \otimes V) \cap (V \otimes I)$:
\begin{enumerate}
\item[(b)] ${\rm Im}(\kappa^L \otimes \id - \id \otimes \kappa^L) \subseteq I$;

\item[(c)] 
$\kappa^L \circ (\kappa^L \otimes \id - \id \otimes \kappa^L) = -(\kappa^C \otimes \id - \id \otimes \kappa^C)$;

\item[(d)] 
$\kappa^C \circ ( \id \otimes \kappa^L - \kappa^L \otimes \id) \equiv 0.$
\end{enumerate}
\end{theorem}

Recall Notation~\ref{not:B,H,kappa}: $B$ is generated by the $k$-vector space $V$
with quadratic relations $I\subset V\ot V$, so that $B=T_k(V)/( I )$.  Moreover, consider:

\begin{notation} \label{not:U,R,P} [$U$, $T_H(U)$, $R$, $P$]
Let $U := V\ot H$, which is an $H$-bimodule under the actions defined in Section~\ref{subsec:actions}.
Set  $R= I\ot H$, similarly an $H$-bimodule, and an $H$-subbimodule of $U \otimes _H U$. Let $P =\{  r\ot 1 - \kappa(r) \mid r \in I \}$ be the relation space of $\mc{D}_{B, \kappa}$,
generating an   $H$-submodule of $H\oplus U \oplus (U\ot _H U)$ in the tensor algebra:
$$T_H(U) = H \oplus U \oplus (U \otimes_H U) \oplus (U \otimes_H U \otimes_H U) \oplus \cdots.$$ Note that $U^{\otimes_H^n} \cong V^{\otimes n} \otimes H$ as $k$-vector spaces. We see that $\pi(P) = R$, where the map $\pi$ is the projection onto the homogeneous quadratic part of $P$. 
\end{notation}

Consider the following preliminary results.

\begin{lemma} \label{lem:T_H(U)} 
Since $T_H(U)$ is canonically isomorphic to $T_k(V) \# H$, we have that
$$T_H(U)/(P) ~\cong~ \mc{D}_{B,\kappa} \quad \text{and} \quad T_H(U)/(R) ~\cong~ (T_k(V) \# H) / (I) ~\cong~ B \#H,$$ where $(I)$ is identified as an ideal of $T_k(V) \# H$, generated by $I$.

Hence, $\mc{D}_{B, \kappa}$ is a PBW deformation of $B \#H$ if and only if $T_H(U)/(P)$ is a PBW deformation of $T_H(U)/(R)$. \qed
\end{lemma}

\begin{lemma} \cite[Lemma~5.2]{SW}  \label{lem:alphabeta}
If $T_H(U)/(P)$ is a PBW deformation of $T_H(U)/(R)$, then the following conditions hold for maps $\alpha: R \rightarrow U$ and $\beta: R \rightarrow H$ for which $P = \{x - \alpha(x) - \beta(x) ~|~ x \in R\}$:
\begin{enumerate}
\item[(i)] ${\rm Im} (\alpha \otimes_H \id - \id \otimes_H \alpha) \subseteq R$;

\item[(ii)] 
$\alpha \circ (\alpha \otimes_H \id - \id \otimes_H \alpha) = -(\beta \otimes_H \id - \id \otimes_H \beta)$;

\item[(iii)] 
$\beta \circ ( \id \otimes_H \alpha - \alpha \otimes_H \id) \equiv 0.$
\end{enumerate}
Here, the maps $\alpha \ot_H \id - \id \ot_H \alpha$ and $\beta \ot_H \id - \id \ot_H \beta$ are defined on the subspace $(R \otimes_H U) \cap (U \otimes_H R)$ of $T_H(U)$. \qed 
\end{lemma}

\begin{remark} \label{rem:alphakappa}
Given the maps $\kappa^L: I \rightarrow V \otimes H$ and $\kappa^C:I \rightarrow H$ as in Notation~\ref{not:B,H,kappa}, we see that $\alpha = \kappa^L \otimes \id_H$ and $\beta= \kappa^C \otimes \id_H$.
\end{remark}

\begin{lemma}  \label{lem:T/P_t}
Consider the algebra 
$$(T_H(U)/(P))_t := \frac{T_H(U)[t]}{(x - \alpha(x)t - \beta(x)t^2)_{x \in R}}.$$
We have that $(T_H(U)/(P))_t$ is a PBW deformation of $T_H(U)/(R)$ over $k[t]$ if and only if $\mc{D}_{B,\kappa,t}$ (of Proposition~\ref{prop:doa equiv}) is a PBW deformation of $B\# H$ over $k[t]$. \qed
\end{lemma}

\begin{proof}
This follows from 
Lemma~\ref{lem:T_H(U)} and Remark~\ref{rem:alphakappa}.
\end{proof}

Now we provide the proof of Theorem~\ref{thm:main}.
A somewhat shorter proof would suffice in case $H$ is semisimple:
The first proof of \cite[Theorem 3.1]{SW:DOA} may be generalized from semisimple group algebras  to  semisimple Hopf algebras. 
In that context, one has on hand a much smaller resolution than that which we
will use below.

\medskip

\noindent {\it Proof of Theorem~\ref{thm:main}}. 
Note that we will employ the identifications given in the lemmas and remark above, sometimes without comment. Namely, results from Section~\ref{sec:resolution} will be used here where, for instance, $I$ is identified 
so that $R= I \ot H$, and $B\# H$ is identified with $T_H(U)/(R)$.
\medskip

\noindent \underline{Necessity of conditions (a)--(d).}
Let us first show that conditions (a)-(d) are necessary.  Assume that $\mc{D}_{B, \kappa}$ is a PBW deformation of $B\# H$ and take $Q$ to be the relation space of $\mc{D}_{B,\kappa}$. Then, for all $h\in H$ and $r\in I$, we have that $h \cdot r - h \cdot (\kappa (r)) \in Q$. (Refer to Section~\ref{subsec:actions} for the definition of these actions.) 
We also have that $h \cdot r -  \kappa (h \cdot r) \in Q$, so $h \cdot (\kappa(r)) - \kappa(h \cdot r)\in Q$. This implies that $h\cdot (\kappa(r)) = \kappa(h\cdot r)$ in $\mc{D}_{B, \kappa}$ since $Q$ cannot contain nonzero elements in degree less than two. Thus, condition (a) holds. 
Moreover, by Lemma~\ref{lem:T_H(U)}, $T_H(U)/(P)$ satisfies the PBW property.

Now by applying Lemma~\ref{lem:alphabeta}, we see that conditions (i),(ii),(iii) hold for $T_H(U)/(P)$. These conditions are equivalent to conditions (b),(c),(d) in Theorem~\ref{thm:main} for the algebra $\mc{D}_{B,\kappa}$ by Notation~\ref{not:U,R,P} and Remark~\ref{rem:alphakappa}. Thus, if $\mc{D}_{B, \kappa}$ is a PBW deformation of $B\# H$, then conditions (a)-(d) of Theorem~\ref{thm:main} must hold.  

\medskip

\noindent\underline{Sufficiency of conditions (a)--(d).} Conversely, let us assume that conditions (a)-(d) of Theorem~\ref{thm:main} hold for the algebra $\mc{D}_{B,\kappa}$. Equivalently by Notation~\ref{not:U,R,P}, Lemma~\ref{lem:T_H(U)}, and Remark~\ref{rem:alphakappa}, we assume the following statements for the algebra $T_H(U)/(P)$:
\begin{itemize}
\item[*] the maps $\alpha$ and $\beta$ are $H$-invariant; and 
\item[*] conditions (i),(ii),(iii) of Lemma~\ref{lem:alphabeta} hold.
\end{itemize}
The goal is to show that $\mc{D}_{B,\kappa}$ is a PBW deformation of $B\#H$ which, by Proposition~\ref{prop:doa equiv}, is equivalent to showing that $\mc{D}_{B,\kappa,t}$ (of Proposition~\ref{prop:doa equiv}) is a graded deformation of $B\# H$ over $k[t]$. Hence, by Lemma~\ref{lem:T/P_t}, the goal is then equivalent to verifying that the algebra $(T_H(U)/(P))_t$ is a graded deformation of $T_H(U)/(R)$ over $k[t]$. We thus have the following strategy:

\begin{itemize} 
\item[-] Let $A$ denote $T_H(U)/(R)$. 
\item[-] Construct multiplication maps, $\mu_i: A \otimes A \rightarrow A$, as in Definition~\ref{def:deformation}, subject to the restraints listed in Remark~\ref{rem:mu}. 
\item[-] Form the graded deformation $A_t$  of $A$ as in Definition~\ref{def:deformation}. 
\item[-] Conclude that $A' := T_H(U)/(P) \cong (A_t)|_{t=1}$ is a PBW deformation of $A$ by Proposition~\ref{prop:doa equiv}.
\end{itemize}

We generalize the proof in \cite{SW} from group actions to Hopf algebra actions. Namely, we use the free resolution $X_{\DOT}$ of the $A^e$-module $A$ in Construction~\ref{const:X} to define the maps $\mu_i$. Recall that $X_{\DOT}$ is constructed from $C_{\DOT}= \mathbf{B}_{\DOT}(A)$, the bar resolution of $H$,
and from $D_{\DOT}$, the Koszul resolution of $B$. 
\medskip 

\noindent {\it Extending $\alpha$ and $\beta$ to be maps on $X_{\DOT}$}. 

We first extend $\alpha$ and $\beta$ to be maps on $X_{\DOT}$ as follows.
In degree 2, $X_2$ contains as a direct summand
$X_{0,2}\cong A\ot I\ot A$; see \eqref{X0n}.
As $\alpha,\beta$ are $H$-bilinear by $H$-invariance, we may extend them to $A^e$-module maps
from $A\ot R\ot _H A\cong A\ot I\ot A$ to $A$
by composing with the multiplication map. 
By {\it abuse of notation}, denote these extended maps by $\alpha, \beta$ as well. 
Extend $\alpha$ and $\beta$ yet further by setting them equal to  
0 on the summands $X_{2,0}$ and $X_{1,1}$ of $X_2$
so that  they become maps $\alpha,\beta: X_2 \rightarrow A$. More precisely, $\alpha, \beta \in \Hom_{A^e}(X_2, A) \cong \Hom_k(X_2', A)$ for $X_2 \cong A \ot X_2' \ot A$.

\medskip

\noindent {\it Construction of the multiplication map $\mu_1$}. 

To build $\mu_1 \in \Hom_k(A\ot A, A) \cong \Hom_{A^e}(A^{\ot 4}, A)$,  recall that it must
be a Hochschild 2-cocycle as in~\eqref{eq:mu1}. We will show that $\alpha: X_2 \rightarrow A$ is a Hochschild 2-cocycle on $X_{\DOT}$, that is, $d_3^*(\alpha) =0$. Recall the chain maps of Lemma~\ref{lem:maps}. 
We set $\mu_1 = \psi_2^*(\alpha)$, which will be a Hochschild 2-cocycle on $B_{\DOT}(A)$, that is, $\delta_3^*(\mu_1)=0$.

To show that $d_3^*(\alpha): X_3 \rightarrow A$ is the zero map, first note that $X_3 = X_{0,3} \oplus X_{1,2} \oplus X_{2,1} \oplus X_{3,0}$ from \eqref{resolution-X} and that the images of $X_{2,1}$ and $X_{3,0}$ under $d_3$ lie in $X_{1,1} \oplus X_{2,0}$. Since $\alpha|_{X_{1,1} \oplus X_{2,0}} \equiv 0$ by the extension above, it suffices to show that $d_3^*(\alpha)|_{X_{0,3}}$ and $d_3^*(\alpha)|_{X_{1,2}}$ are zero maps.

Rewriting condition (i) of Lemma~\ref{lem:alphabeta}, we see that it implies that  $\alpha$ is 0 on the image of the differential on
$X_{0,3}$ as follows. 
Let $\sum_i 1\ot x_i\ot y_i\ot z_i\ot 1 \in A\ot ((I\ot V)\cap (V\ot I))\ot A=X_{0,3}$; see \eqref{X0n}. Then

\[
 \alpha \left(d_3\left(\sum_i 1\ot x_i\ot y_i\ot z_i\ot 1 \right) \right) \hspace{4in}
\]
\[
\hspace{.7in} = \alpha \left(\sum_i  x_i\ot y_i\ot z_i\ot 1 - \sum_i 1 \ot x_i y_i \ot z_i \ot 1 + 
\sum_i 1\ot x_i \ot y_i z_i \ot 1 -\sum_i 1 \ot x_i \ot y_i \ot z_i 
    \right)  
\]
\[ 
   = \sum_i (x_i\alpha(y_i\ot z_i)  -0 +0
   -\alpha(x_i\ot y_i)z_i). \hspace{2.48in}
\]
(To see this, note that applying the multiplication map of $A$ to 
elements in $I$ yields 0.)
Thus $d_3^*(\alpha) = \id\ot\alpha - \alpha\ot\id$ on $X_{0,3}$; here, we identify $\id\ot\alpha - \alpha\ot\id$ with $m \circ (\id\ot\alpha - \alpha\ot\id)$ where $m$ is  the
multiplication map on $A$. We see that condition (i)  indeed implies (in fact, is equivalent to) 
$d^*_3(\alpha)|_{X_{0,3}}\equiv 0$.

Next, we claim that $\alpha$ being $H$-invariant implies that $\alpha$
is also 0 on the image of the differential on $X_{1,2}$. 
Let $a,b\in A$, $h\in H$, and $r\in I$, and
consider $a\ot h\ot r\ot b$ as an element of $X_{1,2}\cong A\ot H\ot I\ot A$ by Theorem~\ref{thm:resolution}(a).
By the definition of the differential on $X_{1,2}$,
\begin{eqnarray*}
   d(a\ot h\ot r\ot b) & = & d((a\ot h\ot 1)\ot (1\ot r\ot b))\\
    &=& d(a\ot h\ot 1)\ot (1\ot r\ot b) - (a\ot h\ot 1)\ot d(1\ot r\ot b).
\end{eqnarray*}
The second term lies in $X_{1,1}$, but 
$\alpha$ is 0 on $X_{1,1}$ by definition. 
Therefore 
\begin{eqnarray} \label{eq:alphacocycle}
\begin{array}{rl}
  \alpha(d(a\ot h\ot r\ot  b)) 
    &= \alpha((ah\ot 1 - a\ot h)\ot (1\ot r\ot b))\\
    &= \alpha\left(ah\ot r\ot b - \sum a\ot  (h_1 \cdot r) \ot h_2b\right)\\
   &= ah\alpha(r) b - \sum a \alpha(h_1 \cdot r) h_2b.
\end{array}
\end{eqnarray}
Since $\alpha$ is $H$-invariant (*), we have that 
$$h\alpha(r) = \sum h_1 \epsilon(h_2) \alpha(r) = \sum h_1 \alpha(r) \epsilon(h_2)= \sum h_1 \alpha(r) S(h_2) h_3 \overset{(*)}{=} \sum \alpha(h_1 \cdot r) h_2.$$
Thus, $\alpha$ is zero on the image of $d=d_3$ on $X_{1,2}$ by \eqref{eq:alphacocycle}. 
It follows that $\alpha$ is a Hochschild 2-cocycle on $X_{\DOT}$.

Now, let $\mu_1=\psi_2^*(\alpha)$, where $\psi_{\DOT}$ is
a chain map satisfying the conditions of Lemma~\ref{lem:maps}.
We conclude that
$$\delta_3^*(\mu_1) = \delta_3^*(\psi_2^*(\alpha)) = \psi_3^*(d_3^*(\alpha)) \equiv 0,$$
as desired.
So, we have a first level graded deformation $A_{(1)}$ of $A$ with 
first multiplication map $\mu_1: A\ot A\rightarrow A$. 
\smallskip

As an aside, we also get that $\phi_2^*(\mu_1)=\alpha$ as cochains. To see this, first note that since $\alpha$ is homogeneous of degree $ - 1$ by its definition,
so is $\mu_1$. 
Let $x\in X_{0,2}$. 
By Lemma~\ref{lem:maps}, $\psi_2\phi_2(x) = x$, and thus
$$
    \mu_1\phi_2(x) = \alpha\psi_2\phi_2(x) = \alpha(x).
$$
Now  let $y$ be a free generator of
$X_{1,1}$ or of $X_{2,0}$, which may always be chosen to have degree 1 or 0, respectively.
Then $\psi_2\phi_2(y)$ has respectively degree 1 or 0, implying that
its component in $X_{0,2}$ is 0. It follows that 
$\mu_1\phi_2(y)= \alpha\psi_2\phi_2 (y) = 0 = \alpha(y)$; the last equation follows from the extension of $\alpha$ to $X_{\DOT}$. Therefore $\phi_2^*(\mu_1)=\alpha$.

\medskip

\noindent {\it Construction of the multiplication map $\mu_2$}. 

Given $\mu_1$ as above, note that the map $\mu_2$ must satisfy \eqref{eq:mu2}, that is, $\delta_3^*(\mu_2) = \mu_1 \circ (\mu_1 \ot \id - \id \ot \mu_1)$ as cochains on the bar resolution $\mathbf{B}_{\DOT}(A)$ of $A$. We will show that a modification of $\psi_2^*(\beta)$ is such a map as follows.

First, note that $\beta = \phi_2^*(\psi_2^*(\beta))$ as cochains by a similar argument to that above for $\alpha$.
Moreover, condition~(ii) implies that 
$d_3^*(\beta)=\alpha\circ (\alpha\ot_H \id - \id\ot_H \alpha)$
as cochains on $X_{0,3}$. 
Let 
\begin{equation} \label{eq:gamma}
\gamma = \delta_3^* \psi_2^*(\beta) - \mu_1\circ (\mu_1\ot\id - \id\ot\mu_1).
\end{equation}
Then $\phi_3^*(\gamma)$ is zero on  $X_{0,3}$:
$\phi_3^* \delta_3^* \psi_2^*(\beta) = d_3^*(\beta)$ and $\phi_3^*(\mu_1 \circ (\mu_1 \ot \id - \id \ot \mu_1)) = \alpha \circ (\alpha \ot \id - \id \ot \alpha)$
by Lemma~\ref{lem:alphabeta}(ii).
To see the last statement, note that the image of $\phi_3$ on $X_{0,3}$ is contained in $A\ot ((I\ot V) \cap (V\ot I))\ot A$ with $\phi^*(\mu_1)=\alpha$.
We also see that $\phi_3^*(\gamma)$ is 0 on $X_{2,1}$ and on $X_{3,0}$
since it is a map of degree $ - 2$. 
We claim  it is also 0 on $X_{1,2}$ as follows. As an 
$A^e$-module, the image of $X_{1,2}$ under $\phi_3$ is generated
by elements of degree 2.
Since $\mu_1=\psi_2^*(\alpha)$, it is zero on elements of degree less than two, and so 
the map $\mu_1\circ (\mu_1\ot\id - \id\ot \mu_1)$ 
must be 0 on the image of $X_{1,2}$ under $\phi_3$.
Since $\beta$ is $H$-invariant so that $\beta$ is a cocycle (see the argument following \eqref{eq:alphacocycle})
 and $\phi_2^* \psi_2^*(\beta)=\beta$, 
we have that $\phi_3^*\delta_3^*\psi_2^*(\beta)=d_3^*\phi_2^*\psi^*_2(\beta)=d_3^*(\beta)$ 
is 0 on $X_{1,2}$. 
Therefore $\phi_3^*(\gamma)$ is 0 on $X_{1,2}$. 

We have shown that $\phi_3^*(\gamma)$ is 0 on all of $X_3$, and so 
$\gamma$ must be a coboundary on the bar resolution $\mathbf{B}_{\DOT}(A)$ of $A$.
Thus there is a 2-cochain $\mu$ of degree $-2$ on the bar resolution  with 
\begin{equation} \label{eq:deltagamma}
\delta_3^* (\mu) =\gamma.
\end{equation}
Consider $\psi_2^*(\beta) - \mu$, yet note that $\phi_2^*(\psi_2^*(\beta)-\mu)$  may not agree with $\beta$ on $X_2$. We need such a statement for the next step of constructing $\mu_3$. 
Now 
$$
  d_3^*\phi_2^*(\mu) = \phi_3^*\delta_3^*(\mu)=\phi_3^*(\gamma) = 0,
$$ so the 2-cochain $\phi_2^*(\mu )$ is a {\em cocycle} on the complex 
$X_{\DOT}$. Thus, $\phi_2^*(\mu)$
lifts to a {\em cocycle} $\mu'$ of degree $-2$ 
on the bar complex $\mathbf{B}_{\DOT}(A)$. 
In other words, $\phi_2^*(\psi_2^*(\beta)-\mu+\mu')$ agrees with $\beta$ on $X_2$. 

Moreover,  $\delta_3^*(\mu') = 0$, and by \eqref{eq:gamma} and \eqref{eq:deltagamma}, we have that $\delta_3^* (\psi_2^*(\beta) - \mu) =
   \mu_1\circ (\mu_1\ot\id - \id\ot\mu_1)$. So,
\begin{equation} \label{eq:delta3*psi2}
   \delta_3^* (\psi_2^*(\beta)-\mu +\mu')-\mu_1\circ (\mu_1\ot\id - \id\ot\mu_1) = 0
\end{equation}
on the bar resolution $\mathbf{B}_{\DOT}(A)$ of $A$.

Thus, set $\mu_2$ equal to $\psi_2^*(\beta)-\mu+\mu'$ and we have maps $\mu_1, \mu_2$ 
to obtain a second level graded deformation $A_{(2)}$ of $A$ extending $A_{(1)}$.

\medskip

\noindent {\it Construction of the multiplication map  $\mu_3$}. 

Recall the restraint on $\mu_3$ given in \eqref{eq:mui}, that is,  
$\mu_3$ is a cochain on $\mathbf{B}_{\DOT}(A)$ whose coboundary is given by $\mu_1 \circ(\mu_2 \ot \id - \id \ot \mu_2) + \mu_2 \circ (\mu_1 \ot \id - \id \ot \mu_1)$. We construct $\mu_3$ as follows.

By \eqref{eq:delta3*psi2} and condition (iii) of Lemma~\ref{lem:alphabeta}, we have that $\mu_2 \circ(\mu_1 \ot \id - \id \ot \mu_1)$ is 0 on the image of $\phi$. 
By degree considerations,
$\mu_1 \circ (\mu_2 \ot \id - \id \ot \mu_2)$ is always 0 on the image of $\phi$. 
Therefore, the obstruction
$$
  \mu_2\circ (\mu_1\ot\id - \id\ot \mu_1) + \mu_1\circ (\mu_2\ot \id
   -\id \ot \mu_2)$$
is a coboundary.
Thus there exists a 2-cochain $\mu_3$, necessarily having  degree $-3$, 
satisfying the restraint given above, and the deformation lifts to the third level.

\medskip

\noindent {\it Construction of the multiplication maps $\mu_i$ for $i \geq 4$}. 

The obstruction for a third level graded deformation $A_{(3)}$ of $A$
to lift to the fourth level lies in $\HH^{3, -4}(A)$ by Proposition~\ref{prop:lifting}.
We apply $\phi_3^*$ to this obstruction  to obtain
a cochain on $X_3$. Since there are no cochains of degree $-4$
on $X_3$ by definition (as it is generated by elements of
degree 3 or less), $\phi_3^*$ applied to the obstruction is automatically zero. Therefore, the deformation may be continued
to the fourth level. Similar arguments show that it can be continued
to the fifth level, and so on. 
\medskip

\noindent {\it Construction of $A_t$}. 

Let $A_t$ be the graded deformation of $A$ that we obtain in this manner [Definition~\ref{def:deformation}].
Then, $A_t$ is the $k$-vector space $A [t]$
with multiplication, for all $a_1, a_2\in A$, 
$$
  a_1*a_2 = a_1a_2 + \mu_1(a_1\ot a_2) t + \mu_2(a_1\ot a_2)t^2 + 
\mu_3(a_1\ot a_2) t^3 + \ldots \ , 
$$
where $a_1 a_2$ is the product in $A$
and each $\mu_i:A\otimes A \rightarrow A$ is a $k$-linear map
of homogeneous degree $-i$.
Now for any $r$ in $R$,
$\mu_1(r)=(\psi_2^*\alpha)(r)$ and $\mu_2(r)=(\psi_2^*(\beta)-\mu+\mu')(r)$ by construction,
and $\mu_i(r)=0$ for $i\geq 3$ since $\deg(r)=2$.

\medskip

\noindent {\it Conclusion that $A':=T_{H}(U)/( P )$ is a PBW deformation of $A=T_H(U)/(R)$}. 

Now we show 
 that $A':=T_{H}(U)/( P ) $ is isomorphic, as a filtered algebra, to the fiber
of the deformation $A_t$ at $t=1$ as follows.
Let $A'' = (A_t)|_{t=1}$.
Then $A''$ is generated by $V$ and $H$ and one thus obtains a 
surjective algebra homomorphism 
$
T_{H}(U)\cong T_k(V)\# H \rightarrow A''.
$
The elements of $P$ lie in the kernel (by definition of $A''$),
and thus this map induces a surjective algebra homomorphism
$
A'=T_{H}(U)/( P ) \twoheadrightarrow A''.
$
This map is in fact an isomorphism of filtered algebras by a
dimension argument in each degree.
Therefore $A'$ is a PBW deformation of $A$, since $A''$ is as such [Proposition~\ref{prop:PBWA_t}].
\qed


\section{Examples} \label{sec:examples}

For our examples, we restrict $k$ to be an algebraically closed field of characteristic zero.
There are many interesting examples, both known and new, in this setting. 
Less is known about
Hopf actions on Koszul algebras and corresponding deformations in positive characteristic. 

As an application of Theorem~\ref{thm:main}, we provide various examples of PBW deformations $\mc{D}_{B,\kappa}$ of smash products $B \#H$; recall Notations~\ref{not:B,H,kappa} and~\ref{def:D Bkap}. We do this by describing  deformation parameter(s) $\kappa = \kappa^C + \kappa^L$ below. In particular, Examples~\ref{ex:CBH},~\ref{ex:H8}, and~\ref{ex:Ha1} involve semisimple Hopf actions, and Examples~\ref{ex:T(n)},~\ref{ex:Hsw}, and~\ref{ex:Uqsl2} involve non-semisimple Hopf actions on (skew) polynomial rings. Recall that skew polynomial rings are Koszul by \cite[Example~4.2.1 and Theorem~4.3.1]{PP}.

\subsection{Semisimple Hopf actions}
We begin by revisiting the well-known PBW deformations of Crawley-Boevey and Holland \cite{CBH}. 

\begin{example} \label{ex:CBH}
Take $H = k\Gamma$, for $\Gamma$ a finite subgroup of $SL_2(k)$, and $B = k[u,v]$. For $g = \begin{pmatrix} a&b\\c&d \end{pmatrix} \in \Gamma$, let the action of $g$ on $B$ be given by
$g \cdot u = au+cv$ and $g \cdot v = bu+dv.$  

By \cite{CBH},  the deformation parameter $\kappa$ of the PBW deformation $\mc{D}_{B,\kappa}$ of $B\# H$ must be in the center of $\Gamma$, which we verify again with Theorem~\ref{thm:main}. We assume here that $\kappa^L \equiv 0$ as in \cite{CBH}.

Since $\dim_k V = 2$, only condition (a)  of Theorem~\ref{thm:main} applies. 
So we have that for all $g \in \Gamma$: $g \cdot (\kappa(uv-vu)) = \kappa(g \cdot(uv-vu))$. Now since the determinant of $g$ is 1, $g \cdot(\kappa(uv-vu)) = \kappa(uv-vu)$, and the image of $\kappa$ lies in the center of $k\Gamma$. 
That is, 
$$\mc{D}_{B,\kappa} = \frac{k \langle u,v \rangle \# k\Gamma}{(uv-vu - \lambda)}$$
is a PBW deformation of $k[u,v] \# k\Gamma$ if and only if $\lambda \in Z(k\Gamma)$.
\end{example}

It is worth pointing out that there are analogues of Crawley-Boevey-Holland algebras when working in positive characteristic; see the work of Emily Norton \cite{EN} 
for some examples that are quite different from those in characteristic zero.

The following two Hopf actions were produced by the first author in joint work with Kenneth Chan, Ellen Kirkman, and James Zhang \cite{CWZ:H8}. We thank Chan, Kirkman, and Zhang for permitting us to use these results.

\begin{example} \label{ex:H8}
Let $H := H_8$ be the unique noncommutative noncocommutative semisimple 8-dimensional Hopf algebra 
\cite{KP, Masuoka}, and let $B$ = $k\langle u,v \rangle/(u^2+v^2)$ (which is isomorphic to the skew polynomial ring $k\langle u,v \rangle/(uv+vu)$).
The Hopf algebra $H_8$ is generated by $x$, $y$, $z$ and the relations are 
$$x^2 = y^2 = 1, \quad xy=yx, \quad zx=yz, \quad zy=xz, \quad z^2 = \frac{1}{2}(1+x+y-xy).$$
The rest of the structure of $H_8$ and left $H_8$-action on $B$ are given by 
\[
\begin{array}{l}
\Delta(x) = x \otimes x, \quad \Delta(y) = y\otimes y, \quad 
\Delta(z) = \frac{1}{2}(1 \otimes 1 + 1 \otimes x + y \otimes 1 - y\otimes x) (z \otimes z),\\
\smallskip
\epsilon(x) = \epsilon(y) = \epsilon(z) = 1, \quad S(x) = x, \quad S(y) = y, \quad S(z) =z,\\
x \cdot u = -u, \quad x \cdot v = v, \quad y \cdot u = u, 
\quad y \cdot v = -v, \quad z \cdot u = v, \quad z \cdot v = u.
\end{array}
\]

Let $r:= u^2 + v^2$ and note that 
$$x \cdot r = r, \quad y \cdot r = r, \quad  z \cdot r = r,$$
so the ideal of relations of $B$, $I = \langle r \rangle$, is $H$-stable. 

Since $\dim_k V = 2$, only condition (a)  of Theorem~\ref{thm:main} applies. 
 We begin by computing $\kappa^C$. Let $\kappa^C(r) = \gamma_0 + \gamma_1 x + \gamma_2 y  + \gamma_3 xy + \gamma_4 z + \gamma_5 xz +\gamma_6 yz + \gamma_7 xyz$ for some scalars $\gamma_i\in k$. 
Since $h \cdot (\kappa^C(r)) = \sum h_1 (\kappa^C(r)) S(h_2)$ (see Section~\ref{subsec:actions}),  both $x \cdot (\kappa^C(r)) = \kappa^C(r)$ and $y \cdot(\kappa^C(r)) = \kappa^C(r)$ imply that $\gamma_7 = \gamma_4$ and $\gamma_6 = \gamma_5.$ Moreover, 
$$z \cdot (\kappa^C(r)) = \gamma_0 + \gamma_2 x + \gamma_1 y + \gamma_3 xy + \gamma_4 z + \gamma_5 xz + \gamma_5 yz + \gamma_4 xyz = \kappa^C(r),$$ which implies that $\gamma_2 = \gamma_1$.
Thus,
\begin{equation} \label{eq:H8kC}
\kappa^C(r) =: g(\gamma_0, \gamma_1, \gamma_3, \gamma_4, \gamma_5) 
= \gamma_0 + \gamma_1 (x +  y) + \gamma_3 xy + \gamma_4 (z + xyz) + \gamma_5 (xz + yz).
\end{equation}

On the other hand, let $\kappa^L(r) =  u  \otimes f + v\ot f' \in V \otimes H$ with $f = \delta_0 + \delta_1 x + \delta_2 y  + \delta_3 xy + \delta_4 z + \delta_5 xz +\delta_6 yz + \delta_7 xyz$ and 
$f' = \delta_0'+\delta_1' x + \delta_2'y +\delta_3'xy + \delta_4'z+\delta_5'xz +\delta_6'yz+\delta_7'xyz$ for some scalars $\delta_i,\delta_i'\in k$. Note that $h \cdot (\kappa^L(r)) = \sum {h_1} \cdot  u \otimes h_2 f S(h_3)
 + \sum {h_1}\cdot v\otimes h_2 f'S(h_3)$ (see Section~\ref{subsec:actions}).
Since $x  \cdot(\kappa^L(r)) = \kappa^L(r)$,  it follows that: 
$$\delta_0 = \delta_1 = \delta_2 = \delta_3 = 0, ~\delta_4 = -\delta_7, ~\delta_5 = -\delta_6 \ \ \
\mbox{and} \ \ \ \delta'_4 = \delta'_7, ~\delta'_5 = \delta'_6.$$
By considering the coefficient of $u$ in the equation $y \cdot (\kappa^L(r)) = \kappa^L(r)$, we now find that $f=0$.
Similarly, by considering the coefficient of $v$ in the equation $y \cdot (\kappa^L(r)) = \kappa^L(r)$, we find that $f'=0$.  Hence, $\kappa^L(r) = 0$.

Thus the deformation parameter $\kappa$ of $\mc{D}_{B, \kappa}$ equals its constant part $\kappa^C$, which depends on five scalar parameters as described above. In short, 
$$\mc{D}_{B, \kappa} = \frac{k\langle u,v \rangle \#H_8}{\left( u^2+v^2 - \kappa (u^2+v^2) \right)}$$ is a PBW deformation of $\left(k \langle u,v \rangle/(u^2 + v^2)\right) \#H_8$ if and only if $\kappa(u^2+v^2) = g(\gamma_0, \gamma_1, \gamma_3, \gamma_4, \gamma_5)$ as given in~\eqref{eq:H8kC}.  This yields a five parameter family of PBW deformations of $B \# H_8$.
\end{example}

\begin{example}  \label{ex:Ha1}
Let $H$ be $H_{a:1}$, one of the 16-dimensional semisimple Hopf algebras classified by Kashina in \cite{Kashina16} and let $B$ be the skew polynomial ring:
$$B = \frac{k\langle t,u,v,w \rangle}{\left( \begin{array}{lll}
r_{tu}:=tu-ut, & r_{tv}:=tv-vt, & r_{tw}:=tw+wt\\
r_{uv}:=uv-vu, & r_{uw}:=uw+wu, &r_{vw}:=vw-wv
\end{array}\right)}.$$
The Hopf algebra $H_{a:1}$ is generated by $x$, $y$, $z$ subject to relations
$$x^4 = y^2 = z^2=1, \quad yx=xy, \quad zx = xyz,  \quad zy=yz.$$
The rest of the structure of $H_{a:1}$ and the left $H_{a:1}$-action on $B$ are given by
\[
\begin{array}{l}
\Delta(x) = x\otimes x, \quad \quad \Delta(y) = y \otimes y, \quad\quad 
\Delta(z) = \frac{1}{2}(1 \otimes 1 + 1 \otimes x^2 + y \otimes 1 - y \otimes x^2)(z \otimes z),\\
\epsilon(x) = \epsilon(y) = \epsilon(z) = 1,\quad \quad S(x) = x^3, \quad\quad S(y) = y, 
\quad \quad S(z) = \frac{1}{2}(1 + x^2 + y -x^2 y)z,
\end{array}
\]
\[
\begin{array}{llllll}
x \cdot t = it, &\quad \quad y \cdot t = -t, &\quad \quad z \cdot t = u,
&\quad \quad  x \cdot u = -iu, &\quad \quad y \cdot u = -u, &\quad \quad z \cdot u = t,\\
x \cdot v = v, &\quad \quad y \cdot v = -v, &\quad \quad z \cdot v = w,
&\quad \quad  x \cdot w = -w, &\quad \quad y \cdot w = -w, &\quad \quad z \cdot w = v,
\end{array}
\]
where $i$ is a primitive fourth root of unity in $k$.
Note that
\[
\begin{array}{llllll}
x \cdot r_{tu} = r_{tu}, & \; \;  x \cdot r_{tv} = i r_{tv}, & \; \; x \cdot r_{tw} = -i r_{tw}, &
\; \; x \cdot r_{uv} = -i r_{uv}, &\; \; x \cdot r_{uw} = i r_{uw}, & \; \; x \cdot r_{vw} = -r_{vw},\\
y \cdot r_{tu} = r_{tu}, & \; \; y \cdot r_{tv} = r_{tv}, & \; \; y \cdot r_{tw} =  r_{tw}, &
\; \; y \cdot r_{uv} =  r_{uv}, &\; \; y \cdot r_{uw} = r_{uw}, & \; \; y\cdot r_{vw} = r_{vw},\\
z \cdot r_{tu} = r_{tu}, & \; \;  z\cdot r_{tv} = r_{uw}, & \; \; z \cdot r_{tw} = r_{uv}, &
\; \; z\cdot r_{uv} =  r_{tw}, &\; \; z \cdot r_{uw} = r_{tv}, & \; \; z \cdot r_{vw} = -r_{vw}.
\end{array}
\]
So, the ideal of relations $I = \langle r_{tu}, r_{tv}, r_{tw}, r_{uv}, r_{uw}, r_{vw} \rangle$ of $B$ is $H$-stable.

Now we compute the possible values $\kappa^C(r) \in H$ for all $r \in I$, under condition (a) of Theorem~\ref{thm:main}. Take $\kappa^C(r) =g(\underline{\gamma}) \in H$ given as follows:
\begin{equation*}
\begin{split}
g(\underline{\gamma}) = & \gamma_0 + \gamma_1 x + \gamma_2 x^2 + \gamma_3 x^3
+ \gamma_4 y + \gamma_5 xy + \gamma_6 x^2 y + \gamma_7 x^3 y\\
& \quad + \gamma_8 z+ \gamma_9 xz + \gamma_{10} x^2 z+ \gamma_{11} x^3 z
+ \gamma_{12} yz + \gamma_{13} xyz + \gamma_{14} x^2 yz + \gamma_{15} x^3 yz,
\end{split}
\end{equation*}
where $\gamma_i\in k$.
Note that $h \cdot g(\underline{\gamma}) = \sum h_1 g(\underline{\gamma})S(h_2)$. With the assistance of Affine, a subpackage of Maxima, we have the following computations:

\begin{eqnarray} \label{eq:g alpha x}
\begin{array}{rl}
\hspace{.5in}x \cdot  g(\underline{\gamma}) &= x ~g(\underline{\gamma})~ x^3\\
 &= \gamma_0 + \gamma_1 x + \gamma_2 x^2 + \gamma_3 x^3
+ \gamma_4 y + \gamma_5 xy + \gamma_6 x^2 y + \gamma_7 x^3 y\\
&   \quad+ \gamma_8 yz+ \gamma_9 xyz + \gamma_{10} x^2 yz+ \gamma_{11} x^3y z
+ \gamma_{12} z + \gamma_{13} xz + \gamma_{14} x^2 z + \gamma_{15} x^3 z ;
\end{array}
\end{eqnarray}
\begin{equation}  \label{eq:g alpha y}
y \cdot  g(\underline{\gamma}) = y ~g(\underline{\gamma})~ y  = g(\underline{\gamma});
\end{equation}
\begin{eqnarray}  \label{eq:g alpha z}
\begin{array}{rl}
\hspace{.5in} z \cdot  g(\underline{\gamma}) &= 
\frac{1}{2}\left(z ~g(\underline{\gamma}) ~S(z) + z~ g(\underline{\gamma})~ S(x^2z) + yz ~g(\underline{\gamma}) ~S(z) - yz ~g(\underline{\gamma}) ~S(x^2z) \right)\\
&= \gamma_0  + \gamma_1 xy + \gamma_2 x^2 + \gamma_3 x^3y + \gamma_4 y + \gamma_5 x + \gamma_6 x^2y + \gamma_7 x^3 \\
& \quad+ \gamma_8 z + \gamma_9 xyz + \gamma_{10} x^2z + \gamma_{11} x^3 yz + \gamma_{12} yz + \gamma_{13} xz + \gamma_{14} x^2 yz + \gamma_{15}x^3z.
\end{array}
\end{eqnarray}

For $r_{tu}$, let $\kappa^C(r_{tu}) = g(\underline{\gamma})$. We have that $x \cdot \kappa^C(r_{tu}) = \kappa^C(r_{tu})$ and $y  \cdot \kappa^C(r_{tu}) = \kappa^C(r_{tu})$ implies that $\gamma_8 = \gamma_{12}, ~\gamma_9 = \gamma_{13}, ~\gamma_{10} = \gamma_{14}, ~\gamma_{11} = \gamma_{15}.$
Moreover, $z \cdot \kappa^C(r_{tu}) = \kappa^C(r_{tu})$ implies that
$\gamma_1=\gamma_5,
~\gamma_3=\gamma_7,
~\gamma_9=\gamma_{13},
~\gamma_{11}=\gamma_{15}.$
Therefore,
\begin{equation} \label{eq:kappa r tu}
\begin{split}
\kappa^C(r_{tu}) &= g(\gamma_0, \gamma_1, \gamma_2, \gamma_3, \gamma_4, \gamma_6, \gamma_8, \gamma_9, \gamma_{10}, \gamma_{11})\\
 &= \gamma_0 + \gamma_1 (x+xy) + \gamma_2 x^2 + \gamma_3 (x^3+x^3 y)
+ \gamma_4 y + \gamma_6 x^2 y \\
& \quad + \gamma_8 (z+yz)+ \gamma_9 (xz+xyz) + \gamma_{10} (x^2 z +x^2yz)+ \gamma_{11} (x^3 z +x^3yz).
\end{split}
\end{equation}

For $r_{vw}$, let $\kappa^C(r_{vw}) = g(\underline{\gamma}')$. We have that $x \cdot \kappa^C(r_{vw}) = -\kappa^C(r_{vw})$ and $y  \cdot \kappa^C(r_{vw}) = \kappa^C(r_{vw})$ implies that $\gamma'_0 = \cdots = \gamma'_7 =0, ~ \gamma'_8 = -\gamma'_{12},~\gamma'_9 = -\gamma'_{13}, ~\gamma'_{10} = -\gamma'_{14}, ~\gamma'_{11} = -\gamma'_{15}.$
Moreover, we have that $z \cdot \kappa^C(r_{vw}) = -\kappa^C(r_{vw})$. So the conditions on $\gamma'_i$ in \eqref{eq:g alpha z} then imply that $\gamma'_i = 0$, for $i =0,\dots,7,8,10,12,14$ with $ \gamma'_9 = -\gamma'_{13},  ~\gamma'_{11} = -\gamma'_{15}$. Thus,
\begin{equation} \label{eq:kappa r vw}
\kappa^C(r_{vw}) ~=~ g(\gamma'_9, \gamma'_{11})
 ~=~ \gamma'_9(xz-xyz) + \gamma'_{11}(x^3z-x^3yz) .
\end{equation}

For $r \neq r_{tu}, r_{vw}$, we have that $x \cdot \kappa^C(r) = \pm i \kappa^C(r)$  implies that $\kappa^C(r) = 0$.

\smallskip

We compute $\kappa^L(r)$ under condition (a) of Theorem~\ref{thm:main}.  Fix $r\in I$ and let 
$$\kappa^L(r) = t\otimes f_t + u\otimes f_u + v\otimes f_v + w\otimes f_w \in V\otimes H , $$
for some $f_t, f_u, f_v, f_w \in H$. Since $y$ is central in $H$ and $y\cdot r = r$
for each relation $r$, we have that
$$
\begin{aligned}
\kappa^L(r) ~=~ y \cdot \kappa^L(r)  
& ~=~ y \cdot t \ot y f_t S(y) + y\cdot u\otimes y f_u S(y) + y\cdot v\otimes y f_vS(y)
  + y\cdot w\otimes yf_wS(y) \\
& ~=~ - t\ot f_t  - u\ot f_u - v\ot f_v - w\ot f_w ~=~ - \kappa^L(r) . 
\end{aligned}
$$ 
Thus, $\kappa^L(r) = 0$.
\smallskip

To finish, we apply to $\kappa$ conditions (b)-(d) of Theorem~\ref{thm:main}. Since $\kappa^L(r) = 0$ for all $r \in I$, only condition (c) is pertinent. Namely, we only need to impose $\kappa^C \otimes \id = \id \otimes \kappa^C$ as maps on $\left(I \otimes V\right) \cap \left(V \otimes I\right)$.
This intersection is a 4-dimensional $k$-vector space with basis:
\[
\begin{array}{rl}
s_{tuv} &:= tuv -tvu - utv + uvt + vtu - vut,\\
s_{tuw}&:= tuw+twu-utw-uwt+wtu-wut,\\
s_{tvw}&:= tvw-twv-vtw+vwt+wtv-wvt,\\
s_{uvw}&:=uvw-uwv-vuw-vwu-wuv+wvu.
\end{array}
\]

Since $\kappa^C(r_{tv}) = \kappa^C(r_{uv}) = 0$, we get that
\begin{equation} \label{eq:s_tuv}
\begin{split}
(\kappa^C \otimes \id - \id \otimes \kappa^C)(s_{tuv})
&= \kappa^C(r_{uv})t - \kappa^C(r_{tv}) u + \kappa^C(r_{tu}) v - t \kappa^C(r_{uv}) + u \kappa^C(r_{tv}) - v \kappa^C(r_{tu})\\
& = \kappa^C(r_{tu}) v  - v \kappa^C(r_{tu}).
\end{split}
\end{equation}
Identify $b \in V$ with $b \# 1 \in A$, and $h \in H$ with $1 \# h\in A$. 
Recall that in $A$, we have $(1\# h) (b \#1) = \sum (h_1 \cdot b) \# h_2$. 
Now by using~\eqref{eq:kappa r tu} and by setting \eqref{eq:s_tuv} equal to 0, we get that
\begin{equation}\label{eq:r tu final}
\kappa^C(r_{tu}) ~=~ g(\gamma_0,  \gamma_2) ~=~ \gamma_0 + \gamma_2 x^2.
\end{equation}
Moreover, 
\[
\begin{array}{rl}
(\kappa^C \otimes \id - \id \otimes \kappa^C)(s_{tuw})
&= -\kappa^C(r_{uw})t + \kappa^C(r_{tw}) u + \kappa^C(r_{tu}) w - t \kappa^C(r_{uw}) + u \kappa^C(r_{tw}) - w \kappa^C(r_{tu})\\
& = \kappa^C(r_{tu}) w - w \kappa^C(r_{tu}) ~=~ 0
\end{array}
\]
imposes no new restrictions on $\kappa^C(r_{tu})$, nor do  
$(\kappa^C \otimes \id - \id \otimes \kappa^C)(s_{tvw}) = 0$, $(\kappa^C \otimes \id - \id \otimes \kappa^C)(s_{uvw}) =0.$
Therefore, $\kappa^C(r_{tu})$ is given by \eqref{eq:r tu final}.
 
To compute $\kappa^C(r_{vw})$, consider the calculation below:
\begin{equation} \label{eq:s_tvw}
\begin{split}
(\kappa^C \otimes \id - \id \otimes \kappa^C)(s_{tvw})
&= \kappa^C(r_{vw})t - \kappa^C(r_{tw}) v + \kappa^C(r_{tv}) w - t \kappa^C(r_{vw}) + v \kappa^C(r_{tw}) - w \kappa^C(r_{tv})\\
& = \kappa^C(r_{vw}) t  - t \kappa^C(r_{vw}).
\end{split}
\end{equation}
Now by using~\eqref{eq:kappa r vw} and by setting \eqref{eq:s_tvw} equal to 0, we get that $\kappa^C(r_{vw}) = 0$.

Therefore, the filtered algebra $\mc{D}_{B, \kappa}$ is a PBW deformation of $B \# H_{a:1}$ if and only if the deformation parameter $\kappa = \kappa^C$ of $\mc{D}_{B, \kappa}$ is defined by \eqref{eq:r tu final} for the relation $r_{tu}$, and $\kappa^C(r)= 0$ for $r \neq r_{tu}$.  Hence, we have a two parameter family of PBW deformations of $B \# H_{a:1}$.
\end{example}

\subsection{Non-semisimple Hopf actions}

Here, we consider non-semisimple Hopf actions to illustrate Theorem~\ref{thm:main}. 
We begin with an example of a Taft algebra action.

\begin{example} \label{ex:T(n)}
Let $H = T(n)$,
the $n^2$-dimensional non-semisimple Taft algebra. We take $n \geq 3$ here and we consider the
(slightly different) case 
$n=2$ (the Sweedler algebra) in Example~\ref{ex:Hsw} below. Let $B$ = $k[u,v]$.
The Hopf algebra $T(n)$ is generated by $g$, $x$ and the relations are 
$g^n= 1$, $x^n = 0$, and $xg=\zeta gx$, for $\zeta \in k^{\times}$ a primitive $n$-th root of unity.
The rest of the structure of $T(n)$ and left $T(n)$-action on $B$ are given by 
\[
\begin{array}{c}
\Delta(g) = g \otimes g, \quad \Delta(x) = g\otimes x + x \otimes 1, 
\quad \epsilon(g) = 1, \quad  \epsilon(x) = 0, \quad S(g) = g^{-1}= g^{n-1}, \quad S(x) = -g^{n-1}x,\\
g \cdot u = u, \quad g \cdot v = \zeta^{-1}v, \quad x \cdot u = 0, 
\quad x \cdot v = u.
\end{array}
\]
Let $r:= uv - vu$ and note that $g \cdot r = \zeta^{-1}r$ and $x \cdot r = 0$.  Hence, the ideal of relations $I = \langle r \rangle$ of $B$ is $H$-stable. 

Since $\dim_k V = 2$, only condition (a)  of Theorem~\ref{thm:main} applies. Now, we compute $\kappa^C$. Let $\kappa^C(r) = \sum_{i,j=0}^{n-1} \gamma_{ij} g^i x^j$. Since $h \cdot (\kappa^C(r)) = \sum h_1 (\kappa^C(r)) S(h_2)$,   for $h \in H$,
we have that $g \cdot (\kappa^C(r)) = \zeta^{-1}\kappa^C(r)$ implies that all terms equal zero except when $j=1$; hence
$$\kappa^C(r) = \gamma_0 x + \gamma_1 gx + \dots + \gamma_{n-1} g^{n-1}x,$$
for $\gamma_i \in k$.
Also, $x \cdot (\kappa^C(r)) = 0$ implies that all terms equal zero except when $i = n-1$, so  
\begin{equation} \label{eq:T(n)kC}
\kappa^C(r) = \gamma g^{n-1}x
\end{equation}
for $\gamma \in k$.

On the other hand, let $\kappa^L(r) =  u  \otimes  f +  v\otimes f'  \in V \otimes H$, for $f=\sum_{i,j=0}^{n-1} \lambda_{ij}g^i x^j$ and
$f'=\sum_{i,j=0}^{n-1}\lambda'_{ij}g^ix^j$. Note that $h \cdot (\kappa^L(r)) = \sum {h_1} \cdot  u \otimes h_2 f S(h_3) + \sum {h_1}\cdot  v\otimes h_2 f' S(h_3)$ (see Section~\ref{subsec:actions}). Since $g \cdot (\kappa^L(r)) = \zeta^{-1}\kappa^L(r)$, all terms equal zero except possibly those in the first sum for which
$j=1$ and those in the second sum for which  $j =0$. 
Therefore 
$\kappa^L(r) ~=~ u \otimes f +v\otimes f' $
for  $$f= \lambda_0 x + \lambda_1 g x + \cdots + \lambda_{n-1} g^{n-1} x
\quad \text{and} \quad
f' = \lambda_0' + \lambda_1' g + \cdots + \lambda_{n-1}'g^{n-1}$$
 with
 $\lambda_i, \lambda_i' \in k$.
Applying $x$, we obtain 
\[
\begin{array}{rl}
0 = x \cdot \kappa^L(r) &= (g \cdot u)  \otimes \left(g f S(x) + xfS(1)\right) + (x \cdot u) \otimes f +(g\cdot v) \ot (g f' S(x) + xf' S(1)) + (x\cdot v) \ot f'  \\
&=u \otimes(-gfg^{n-1}x + xf +f') + \zeta^{-1}v\ot (-gf'g^{n-1}x + xf'). 
\end{array}
\]
It follows that 
$$ -gfg^{n-1}x + xf + f'=0 \ \ \ \mbox{and} \ \ \ -gf'g^{n-1}x + xf'=0.$$
Since $f'$ is in the group algebra $kG(T(n)) \cong k \mathbb{Z}_n$ and $g^n=1$, the second equation implies
that $xf' = f'x$, and so $f' = \lambda_0$ is constant. The first equation
further implies that $f'=0$ and 
that all terms of $f$ are equal to zero except when $i = n-1$. Thus,
\begin{equation} \label{eq:T(n)kL}
\kappa^L(r) ~=~ u \otimes \lambda g^{n-1} x,
\end{equation}
for $\lambda \in k$. 

In summary, 
$$\mc{D}_{B, \kappa} = \frac{k\langle u,v \rangle \#T(n)}{\left( uv-vu - \kappa (uv-vu) \right)}$$ is a PBW deformation of $k[u,v] \#T(n)$ if and only if the deformation map $\kappa$ equals $\kappa^C + \kappa^L$ as given in \eqref{eq:T(n)kC} and~\eqref{eq:T(n)kL}. So, we have a two parameter family of PBW deformations of $k[u,v] \# T(n)$.
\end{example}

\begin{example} \label{ex:Hsw}
Let $H$ be $H_{Sw} = T(2)$,
the $4$-dimensional non-semisimple Sweedler algebra, which is a Taft algebra with $n=2$. Let $B$ = $k[u,v]$.
Retaining the notation from Example~\ref{ex:T(n)},  the Hopf algebra $H_{Sw}$ is generated by $g$, $x$ and acts on $B$ by 
$g \cdot u = u, ~ g \cdot v = -v, ~x \cdot u = 0, ~ x \cdot v = u.$
Similar to Example~\ref{ex:T(n)}, let $r:= uv - vu$ and note that $g \cdot r = -r$ and $x \cdot r = 0$. So, $I = \langle r \rangle$ is $H$-stable.

 Let $\kappa^C(r) = \gamma_0 + \gamma_1 g + \gamma_2 x + \gamma_3 gx$. We have that $g \cdot (\kappa^C(r)) = -\kappa^C(r)$ implies that $\gamma_0 = \gamma_1 = 0$. 
 Moreover, $x \cdot (\kappa^C(r)) = 0$ does not yield new restrictions on $\kappa^C(r)$. Thus, for $\gamma, \gamma' \in k$, we get that
$\kappa^C(r) = \gamma x + \gamma' gx.$
In the same fashion as Example~\ref{ex:T(n)}, we also get that
$\kappa^L(r) = u \otimes (\lambda x + \lambda' gx)$,
for $\lambda$, $\lambda'$ $\in k$.

In summary, 
$$\mc{D}_{B, \kappa} = \frac{k\langle u,v \rangle \#H_{Sw}}{\left( uv-vu - \kappa (uv-vu) \right)}$$ is a PBW deformation of $k[u,v] \#H_{Sw}$ if and only if the deformation map $\kappa$ equals $\kappa^C + \kappa^L$, where
$$\kappa^C(uv-vu) = \gamma x+ \gamma' gx \quad \text{and} \quad \kappa^L(uv-vu) = u \ot (\lambda x + \lambda'gx)$$ for $\gamma, \gamma', \lambda, \lambda' \in k$. Thus, we have a four parameter family of PBW deformations of $k[u,v] \# H_{Sw}$.
\end{example}

\begin{remark}
The invariant ring resulting from the action of $H_{Sw}$ on $k[u,v]$ is isomorphic to the polynomial ring $k[u,v^2]$, that is to say, $k[u,v]^{H_{Sw}}$ is regular. Recall that the Shephard-Todd-Chevalley Theorem states that when given a finite group ($G$-) action on a commutative polynomial ring $R$ that is linear and faithful, $R^G$ is regular if and only if $G$ is a reflection group. Our results would then suggest that 
$H_{Sw}$ is a ``reflection Hopf algebra.'' Ram and  Shepler showed that there are no non-trivial PBW deformations of $k[v_1, \dots, v_n] \# kG$ for many complex reflection groups $G$; such deformations are referred to as {\it graded Hecke algebras} \cite{RS}. Now by broadening their setting to Hopf actions on (possibly noncommutative) regular algebras, we consider new objects in representation theory: Hopf analogues of graded Hecke algebras. Non-trivial examples of these objects exist as we showed in the example above. Further examples and a general explanation of this phenomenon are worthy of further investigation.
\end{remark}

Now we consider the well-known Hopf action of $U_q(\mathfrak{sl}_2)$ on $k\langle u,v \rangle/(uv-qvu)$, where $q \in k^{\times}$ with $q^2 \neq 1$. A PBW deformation of $\left(k\langle u,v \rangle/(uv-qvu)\right) \# U_q(\mathfrak{sl}_2)$ was studied by Gan and Khare in \cite{GK}; we recover their result below. Such algebras are known as {\it quantized symplectic oscillator algebras of rank 1}.

\begin{example} \label{ex:Uqsl2}
Fix $q \in k^{\times}$, with $q^2 \neq 1$.
Let $H$ be the Hopf algebra $U_q(\mathfrak{sl}_2)$, and $B$ = $k\langle u,v \rangle/(uv-qvu)$. As in \cite[I.6.2]{book:BrownGoodearl}, we take $U_q(\mathfrak{sl}_2)$ to be generated by $E, F, K, K^{-1}$ with relations:
$$EF-FE=(q-q^{-1})^{-1}(K - K^{-1}), \quad \quad KEK^{-1} = q^2 E, \quad \quad KFK^{-1} = q^{-2}F, \quad \quad KK^{-1} = K^{-1}K = 1.$$
So, $U_q(\mathfrak{sl}_2)$ has a $k$-vector space basis $\{E^i F^j K^m\}_{i,j \in \NN; m \in \ZZ}$. The rest of the structure of $U_q(\mathfrak{sl}_2)$ and left $U_q(\mathfrak{sl}_2)$-action on $B$ is given by:
\[
\begin{array}{llll}
\Delta(E) = E \ot 1 + K \ot E, \quad &\Delta(F) = F \ot K^{-1} + 1 \ot F, \quad &\Delta(K) = K \ot K,\quad  &\Delta(K^{-1}) = K^{-1} \ot K^{-1},\\
\epsilon(E) = 0, ~&\epsilon(F) = 0 ~ &\epsilon(K) = 1, ~&\epsilon(K^{-1}) =1,\\
S(E) = -K^{-1}E, ~ &S(F) = -FK, ~ &S(K) = K^{-1}, ~ &S(K^{-1}) =K,
\end{array}
\]

\vspace{-.1in}

\[
\begin{array}{llll}
E \cdot u = 0, ~\quad  \quad &F \cdot u =v, ~\quad \quad & K \cdot u = qu, ~ \quad\quad  &K^{-1} \cdot u = q^{-1}u,\\
E \cdot v = u, ~ \quad \quad &F \cdot v =0, ~\quad \quad & K \cdot v = q^{-1}v, ~ \quad \quad &K^{-1} \cdot v = qv.
\end{array}
\]
Let $r:=uv-qvu$ and note that $E \cdot r = F \cdot r =0$ and $K \cdot r = K^{-1} \cdot r =r$. Hence, the ideal of relations $I = \langle r \rangle$ of $B$ is $H$-stable. 

Since $\dim_k V$=2, only condition (a) of Theorem~\ref{thm:main} applies.
Let us compute $\kappa^C(r)$. Since $K \cdot \kappa^C(r) = \kappa^C(K \cdot r) = \kappa^C(r)$, we have that $K \kappa^C(r) S(K)  = \kappa^C(r)$ (see Section~\ref{subsec:actions}). Hence, $K \kappa^C(r) = \kappa^C(r) K$. Moreover,
$$0 ~=~ \kappa^C(E \cdot r) ~=~ E \cdot \kappa^C(r) ~=~ E \kappa^C(r)S(1) + K \kappa^C(r) S(E),$$
so $E \kappa^C(r) = \kappa^C(r) E$. Likewise, $F \cdot \kappa^C(r) = 0$ implies that $F \kappa^C(r) = \kappa^C(r) F$. So, $\kappa^C(r)$ is in the center of $U_q(\mathfrak{sl}_2)$.
For $q$ a non-root of unity, the center of $U_q(\mathfrak{sl}_2)$ is generated by the {\it quantum Casimir element} \cite[Theorem~VI.4.8]{Kassel}, 
\[
C_q ~=~ EF + \frac{q^{-1}K + qK^{-1}}{(q-q^{-1})^2} ~=~ FE + \frac{qK+q^{-1}K^{-1}}{(q-q^{-1})^2},
\]
whereas for $q$ a root of unity,  the elements $E^e$, $F^e$, $K^e$ also belong to the center of $U_q(\mathfrak{sl}_2)$, where  $e=$ord($q^2$).

To compute $\kappa^L(r)$, let $\kappa^L(r) =  u  \otimes \sum \gamma_{ijm} E^i F^j K^m + v\otimes  \sum \gamma'_{ijm} E^iF^jK^m$ for $ \gamma_{ijm}, \gamma'_{ijm} \in k$. Then, 
\[
\begin{array}{rl} 
\kappa^L(r) ~&=~ \kappa^L(K \cdot r) ~=~ \sum K \cdot  u  \otimes \gamma_{ijm}K( E^i F^j K^m )K^{-1} + \sum K\cdot v\otimes \gamma'_{ijm}K(E^iF^jK^m)K^{-1}\\ 
~&=~ \sum q u  \otimes  \gamma_{ijm} q^{2(i-j)}E^i F^j K^m + \sum q^{-1}v\otimes
   \gamma_{ijm}' q^{2(i-j)}E^iF^jK^m\\
 ~&=~
\sum  q^{2(i-j)+1} u  \otimes \gamma_{ijm} E^i F^j K^m + \sum q^{2(i-j)-1}v\otimes \gamma_{ijm}'E^iF^jK^m.
\end{array}
\]
Thus given $m \in \ZZ/ n\ZZ$, define the subspace $V_m \subset U_q(\mathfrak{sl}_2)$ to be the $k$-span of all monomials $E^iF^jK^{\ell}$ such that $j-i \equiv m \mod n$. Then 
$\kappa^L(uv-qvu) \in u \ot V_{2^{-1}} + v \ot V_{-2^{-1}}$ if $q$ is a primitive root of unity of odd order, and $\kappa^L(uv-qvu) =0$ otherwise.

Therefore, 
$$\mc{D}_{B, \kappa} = \frac{k\langle u,v \rangle \# U_q(\mathfrak{sl}_2)}{(uv-qvu-\kappa(uv-qvu))}$$
is a PBW deformation of $(k\langle u,v \rangle/(uv-qvu)) \# U_q(\mathfrak{sl}_2)$ if and only if $\kappa=\kappa^C+\kappa^L$ where $\kappa^C(uv-qvu)$ is in the center of $U_q(\mathfrak{sl}_2)$ and $\kappa^L(uv-qvu)$ is given as above.
\end{example}

More generally, there is a standard $U_q(\mathfrak{sl}_n)$-action on a $q$-polynomial ring $B$ in $n$ variables. 
\begin{question}
Are there nontrivial PBW deformations of the resulting smash product algebra $B \# U_q(\mathfrak{sl}_n)$? 
\end{question}
These would be {\it quantized symplectic oscillator algebras of rank $n-1$}, and   merit further investigation.

\section*{Acknowledgments}
Part of this work was completed during the Spring 2013 program on Noncommutative Algebraic Geometry and Representation Theory at the Mathematical Sciences Research Institute. We thank the organizers for the facilities and generous hospitality. Walton and Witherspoon are supported by the National Science Foundation, grants \#DMS-1102548 and \#DMS-1101399, respectively.

\bibliography{HopfDOA_biblio}

\end{document}